\documentclass[a4paper,12pt]{amsart}
\usepackage{amsthm,amsmath, amssymb, amscd,mathtools, braket}
\usepackage[top=30truemm, bottom=30truemm, left=25truemm, right=25truemm]{geometry}
\usepackage{color}
\usepackage{stmaryrd}
\usepackage{caption}
\usepackage{url}

\allowdisplaybreaks[4]

\usepackage{comment}

\numberwithin{equation}{section}
\newcounter{count}

\newcommand{\num}{\stepcounter{count}\the\value{count}}
\newcommand{\ichi}{ \mbox{\textup{1}}\hspace{-0.26em}\mbox{\textup{l}}} 

\newtheorem{theorem}{Theorem}[section]
\newtheorem{lemma}[theorem]{Lemma}
\newtheorem{corollary}[theorem]{Corollary}

\newtheorem{proposition}[theorem]{Proposition}

\theoremstyle{definition}
\newtheorem{remark}[theorem]{Remark}
\newtheorem{notation}[theorem]{Notation}

\makeatletter
\def\@seccntformat#1{\csname the#1\endcsname.\quad}
\makeatother

\title[ Normality and the Riemann zeta function]{Normality of algebraic numbers and \\the Riemann zeta function}


\author[Y. Kanado]{Yuya Kanado}
\address{Yuya Kanado\\
	Graduate School of Mathematics\\ Nagoya University\\ Furo-cho\\ Chikusa-ku\\ Nagoya\\ 464-8602\\ Japan}
\email{m21017a@math.nagoya-u.ac.jp}

\author[K. Saito]{Kota Saito}
\address{Kota Saito\\Faculty of Pure and Applied Sciences\\ University of Tsukuba\\ 1-1-1 Tennodai\\ Tsukuba\\ Ibaraki\\ 305-8577\\ Japan}
\email{saito.kota.gn@u.tsukuba.ac.jp}

\begin{document}

\begin{abstract}
A real number is called simply normal to base $b$ if every digit $0,1,\ldots ,b-1$ should appear in its $b$-adic expansion with the same frequency $1/b$. A real number is called normal to base $b$ if it is simply normal to every base $b, b^2, \ldots$. In this article, we discover a relation between the normality of algebraic numbers and a mean of the Riemann zeta function on vertical arithmetic progressions. Consequently, we reveal that a positive algebraic irrational number $\alpha$ is normal to base $b$ if and only if we have
\[
\lim_{N\to \infty}\frac{1}{\log N} \sum_{1\leq |n|\leq N}  \zeta\left(-k+\frac{2\pi i n}{\log b} \right) \frac{e^{2\pi i n \log \alpha /\log b}}{n^{k+1}} =0
\] 
for every integer $k\geq 0$.
\end{abstract}

\maketitle

\section{Introduction}

Let $x$ be a positive irrational number, and let $b$ be an integer not less than $2$. Let $x= \sum_{i=-m}^\infty c_i b^{-i}$ be the $b$-adic expansion of $x$, and hence each $c_i$ belongs to $\{0,1,\ldots, b-1\}$. For all $l\in \mathbb{Z}_{\ge 0}$ and $a\in \{0,\ldots, b-1\}$, we define 
\[
A_b(l; a,x)= \#\{i \in [0,l] \colon c_i=a  \}.   
\]
We say that $x$ is \textit{simply normal} to base $b$ if for all $a\in \{0,1,\ldots, b-1\}$, 
\[
\lim_{l\to \infty} A_b(l; a,x)/l= 1/b. 
\]
Further, we say that $x$ is \textit{normal} to base $b$ if $x$ is simply normal to base $b^k$ for every $k\in \mathbb{Z}_{>0}$. Borel introduced the notions of simple normality and normality in \cite{Borel1909}\footnote{Our definition of ``normal" is different from the original one, but these are equivalent (See \cite[Definition~4.1, Theorem~4.2]{Bugeaud2012}). 
}, and in the same article, he showed that almost all real numbers are normal to any base. A simple example was given by Champernowne \cite{Champernowne}. He proved that $0.1\:2\:3\:4\:5\:6\:7\:8\:9\:10\:11\:12\cdots$ is normal to base $10$. Borel \cite{Borel1950} conjectured that all real algebraic irrational numbers are normal in any base at the same time. However, we do not get any examples of normal algebraic numbers. It is very difficult to determine whether a given non-artificial number is normal or not. Actually,  Bailey, Borwein, Crandall, and Pomerance \cite{BBCP} presented the best-known lower bounds for $A_2(l; 1 ,|\alpha| )$ up to constants. They revealed that for every algebraic irrational number $\alpha$ of degree $D$ there exists a constant $C=C(\alpha)>0$,
\[
A_2 (l; 1 ,|\alpha| )\geq C N^{1/D}
\]
for sufficiently large $N$. Thus, our knowledge does not leach to verify the simple normality of $\alpha$ to base $2$.

In recent research, the authors \cite{SK} discovered a strong relation between the simple normality of $2^{p/q}$ and the Riemann zeta function. In \cite[Theorem~2.1]{SK}, they  demonstrated that for all co-prime integers $p$ and $q$ with $1\leq p<q$ and for all $b\in \mathbb{Z}_{\ge 2}$ which is not a $q$-th power of an integer,  
\begin{equation}\label{eq:previous-result}
\sum_{0\leq d\leq l} \{b^{d+p/q} \}  =\frac{l}{2}-\frac{1}{2\pi i}\sum_{1\leq |n|\leq b^l}  \zeta\left(\frac{2\pi i n}{\log b} \right) \frac{e^{2\pi i n p /q}}{n} +o_{p,q,b}(l) \quad  \text{as }l\to \infty,
\end{equation}
where $\zeta(s)$ denotes the Riemann zeta function and $\{x\}$ denotes the fractional part of $x$. Furthermore, by substituting $b=2$, they \cite[Theorem~1.1]{SK} concluded that $2^{p/q}$ is simply normal to base $2$ if and only if the sum on the right-hand side of \eqref{eq:previous-result} has a magnitude of $o(l)$. In this article, by using the periodic Bernoulli polynomials, we successfully extend our result from the simple normality of $2^{p/q}$ to base $2$ to the normality of a general algebraic number to any base.  To assert the main result, we shall introduce the periodic Bernoulli polynomials.  

Let $\mathbb{Q}[x]$ be the set of all polynomials with rational coefficients. We define the Bernoulli polynomials $B_0(x), B_1(x),\ldots \in \mathbb{Q}[x]$ by the coefficients of the Maclaurin expansion of the following generating function:
\[
\frac{te^{xt}}{e^{t}-1} = \sum_{k=0}^\infty \frac{B_k(x)}{k!} t^k. 
\]
For example, $B_0(x)=1$, $B_1(x)=x-1/2$, $B_2(x)=x^2-x+1/6$. For every $k\in \mathbb{Z}_{\ge 0}\setminus \{1\}$, we define $\psi_k(x)= B_k(\{x\})$, and $\psi_k(x)$ is called the \textit{$k$-th periodic Bernoulli polynomial}. For $k=1$, we set
\begin{equation*}
\psi_1(x)=\left\{
\begin{array}{cl}
\{x\}-1/2 & \text{if } x\notin \mathbb{Z}, \\
0 & \text{otherwise.}
\end{array} \right.
\end{equation*}
We refer \cite{ArakawaIbukiyamaKaneko} to the readers who want to know more details. 

\begin{theorem}{\label{Theorem-fractional-powers}}Let $\alpha$ be a positive algebraic irrational number, and $b\in \mathbb{Z}_{\ge 2}$. For every $k\in \mathbb{Z}_{\ge 0}$, we have
\begin{align*}
&\frac{(k+1)!}{(2\pi i)^{k+1}}\sum_{1\leq |n|\leq N} \zeta\left(-k+\frac{2\pi in}{\log b}\right) \frac{e^{2\pi i n \log \alpha /\log b  }}{n^{k+1}}\\
&= -\frac{1}{\log^k b}  \sum_{0<h<\frac{\log N}{\log b}} \psi_{k+1}( b^{h}\alpha  )+o_{k,b,\alpha}(\log N ) \quad \text{as } N\to\infty.
\end{align*}
\end{theorem}
We can reproduce \eqref{eq:previous-result} from  Theorem~\ref{Theorem-fractional-powers} with $k=0$, $\alpha=b^{p/q}$, and $N=b^l$. Thus, Theorem~\ref{Theorem-fractional-powers} is an extension of \eqref{eq:previous-result}. We will complete the proof of the theorem in Section~\ref{GenerelResults}. This is a goal of this article since we obtain the following equivalence.

\begin{corollary}\label{corollary:normal}
Let $\alpha$ be a positive algebraic irrational number. Let $b\geq 2$ be an integer. Then $\alpha$ is normal to base $b$ if and only if for every integer $k\geq 0$,
\begin{align*}
\lim_{N\to \infty}\frac{1}{\log N} \sum_{1\leq |n|\leq N} \zeta\left(-k+\frac{2\pi in}{\log b}\right) \frac{e^{2\pi i n \log \alpha /\log b  }}{n^{k+1}}= 0.
\end{align*}
\end{corollary}

\begin{proof}
Let $x_h = \alpha b^h$ for every $h\in \mathbb{Z}_{>0}$. By \cite[Theorem~4.14]{Bugeaud2012}, $\alpha$ is normal to base $b$ if and only if $(x_h)_{h=1}^\infty$ is uniformly distributed modulo $1$. Further, it is equivalent to get
\begin{equation}\label{eq:uniform-dist}
\lim_{H\to \infty } \frac{1}{H}\sum_{h=1}^H f(\{x_h\}) = \int_{0}^1 f(x) dx   
\end{equation}
for every real continuous function $f$ on $[0,1]$ (See \cite[Theorem~1.1]{KN}). By the Weierstrass approximation theorem,  $\alpha$ is normal to base $b$ if and only if \eqref{eq:uniform-dist} holds for all polynomials on $[0,1]$ with real coefficients. These polynomials can be represented as a finite linear combination of $\{\psi_0(x),\psi_1(x),\psi_2(x),\ldots\}$ over $\mathbb{R}$ since each $\psi_k$ is a polynomial of degree $k$.  By the property of $B_{k+1}(x)$ \cite[p.55]{ArakawaIbukiyamaKaneko}, we have $\int_{0}^1 \psi_{k+1}(x) dx=0$ for every $k\in \mathbb{Z}_{\ge 0}$. Therefore,  $\alpha$ is normal to base $b$ if and only if 
\[
\lim_{H\to \infty } \frac{1}{H}\sum_{h=1}^H \psi_{k+1}(x_h) =  0   
\]
 for all $k\in \mathbb{Z}_{\ge 0}$. By this equivalence and Theorem~\ref{Theorem-fractional-powers}, we obtain Corollary~\ref{corollary:normal}.
\end{proof}

Let $p$ and $q$ be co-prime integers, and let $b\in \mathbb{Z}_{\ge 2}$ which is not a $q$-th power of an integer. By substituting $\alpha =b^{p/q}$ in Corollary~\ref{corollary:normal}, we see that $b^{p/q}$ is normal to base $b$ if and only if for all $k\geq 0$,
\[
\lim_{N\to \infty}\frac{1}{\log N} \sum_{1\leq |n|\leq N} \zeta\left(-k+\frac{2\pi in}{\log b}\right) \frac{e^{2\pi i n p /q  }}{n^{k+1}}= 0.
\]

\begin{comment}
On the other hand, we obtain an upper bound of the left-hand side in Theorem~\ref{Theorem-fractional-powers} using the explicit estimation $|B_k(x)|\leq \frac{2k!}{(2\pi)^k} (1+\ichi(k\in2+4\mathbb{Z})(\zeta(k)-1))$ for $k\geq1$ and $0\leq x\leq 1$ where $\ichi(P)=1$ if $P$ is true;   $\ichi(P)=0$ otherwise for every statement $P$.

\begin{corollary}
    Let $\alpha$ be a positive algebraic irrational number. Let $b$ and $k$ be an integer with $b\geq2$ and $k\geq0$. Then, we have
    \[
    \limsup_{N\to\infty} \frac{1}{\log N} \left| \sum_{1\leq |n|\leq N} \zeta\left( -k+\frac{2\pi in}{\log b} \right) \frac{e^{2\pi in\log\alpha/\log b}}{n^{k+1}}\right| \leq \frac{2+\ichi(k\in1+4\mathbb{Z})(2\zeta(k+1)-2)}{(\log b)^{k+1}}.
    \]
\end{corollary}
\begin{proof}
By Theorem~\ref{Theorem-fractional-powers}, we have
\begin{align*}
    &\quad\limsup_{N\to\infty} \frac{1}{\log N} \left| \sum_{1\leq |n|\leq N} \zeta\left( -k+\frac{2\pi in}{\log b} \right) \frac{e^{2\pi in\log\alpha/\log b}}{n^{k+1}}\right| \\
    &\leq \limsup_{N\to\infty} \frac{1}{\log N} \left| \frac{(2\pi)^{k+1}}{(\log b)^k(k+1)!} \sum_{0<h<\frac{\log N}{\log b}} \psi_{k+1}(b^h\alpha) \right| \\
    &\leq \frac{(2\pi)^{k+1}}{(\log b)^{k+1}(k+1)!} \cdot \frac{2(k+1)!}{(2\pi)^{k+1}} (1+\ichi(k\in1+4\mathbb{Z}) (\zeta(k+1)-1))\\
    &=\frac{2+\ichi(k\in1+4\mathbb{Z}) (2\zeta(k+1)-2)}{(\log b)^{k+1}}.
\end{align*}
\end{proof}
\end{comment}

\begin{notation}
 For every $m\in \mathbb{Z}$, we define $\mathbb{Z}_{\ge m}$ as the set of integers not less than $m$, and we also define $\mathbb{Z}_{>m}$ in a similar manner. For $x\in \mathbb{R}$, let $\|x\|$ denote the distance from $x$ to the nearest integer. Let $\log_b x$ be $\log x/\log b$ for all $x>0$ and $b >0$. We say that $f(x)=g(x)+o(h(x))$ (as $x\to \infty$) if for all $\epsilon>0$ there exists $x_0>0$ such that $|f(x)-g(x)|\leq h(x)\epsilon $ for all $x\geq x_0$. If $x_0$ depends on some parameters $\epsilon, a_1,\ldots ,a_n$, then we write $f(x)=g(x)+o_{a_1,\ldots , a_n}(h(x))$. We also say that $f(x)=g(x)+O(h(x))$ for all $x\geq x_0$ if there exists $C>0$ such that $|f(x)-g(x)|\leq Ch(x)$  for all $x\geq x_0$. If $C$ depends on some parameters $a_1,\ldots , a_n$, then we write $f(x)=g(x)+O_{a_1,\ldots ,a_n}(h(x))$ for all $x\geq x_0$.
\end{notation}

\section{General Results and proof of Corollaries}\label{GenerelResults}

In this section, we give more general results than the ones in the previous section. The following two theorems imply Theorem~\ref{Theorem-fractional-powers}. Therefore, the goal of this article is to give proofs of Theorems~\ref{Theorem-Bernoulli0} and \ref{Theorem-Bernoulli1}. 

\begin{theorem}\label{Theorem-Bernoulli0} Let $d$ be a positive real number, and let $\theta$ be a real number. Then, for every $N\in \mathbb{Z}_{\geq 2}$, we have
\begin{align} \nonumber
&\frac{1}{2\pi i}\sum_{1\leq |n|\leq N} \zeta(2\pi idn) \frac{e^{2\pi i\theta n}}{n}\\ \label{eq:thm-Bernoulli0-1}
&= \sum_{0<h<d\log N} \sum_{1\leq n\leq dNe^{-(h+\theta)/d}} \frac{\sin(2\pi n e^{(h+\theta)/d} ) }{n\pi } + O((\log N)^{2/3} (\log \log N)^{5/9} )\\ \label{eq:thm-Bernoulli0-2}
&=-  \sum_{0<h<d \log N} \psi_1( e^{(h+\theta)/d})+ G_N+O((\log N)^{2/3} (\log \log N)^{5/9}),
\end{align}
where
\[
|G_N|\leq \sum_{0<h<d\log N } \min \left(\frac{1}{2}, \frac{1}{2\pi \|e^{(h+\theta)/d} \| (2dNe^{-(h+\theta)/d}+1) } \right)  .
\]
\end{theorem}

\begin{theorem}{\label{Theorem-Bernoulli1}}Let $k\in \mathbb{Z}_{>0 }$, and let $d>0$ and $\theta$ be real numbers. Then for every integer $N\geq 2$,  
we have
\begin{align*}
&\frac{(k+1)!}{(2\pi i)^{k+1}}\sum_{1\leq |n|\leq N} \zeta(-k+2\pi idn) \frac{e^{2\pi i\theta n}}{n^{k+1}}\\
&= -d^k  \sum_{0<h<d \log N} \psi_{k+1}(  e^{(h+\theta)/d})+O_{k,d,\theta}(\log \log N ).
\end{align*}
\end{theorem}
We will give proofs of Theorems~\ref{Theorem-Bernoulli0} and \ref{Theorem-Bernoulli1} in Section~\ref{section:Proof_of_Theorems_2.1_and_2.2}. By substituting $d=1/\log b$ and $\theta= \log \alpha / \log b$ in Theorem~\ref{Theorem-Bernoulli1}, we immediately obtain Theorem~\ref{Theorem-fractional-powers} with $k>0$. Furthermore, Theorem~\ref{Theorem-Bernoulli0} implies  Theorem~\ref{Theorem-fractional-powers} with $k=0$. To prove this implication, we apply the following theorem.

\begin{theorem}[Ridout's theorem]\label{Theorem:Ridout}Let $\alpha$ be any algebraic irrational number. 
Let $Q_1,\ldots , Q_t$ be distinct primes. Let $b$ be an integer of the form 
\begin{equation}\label{equation:primefactors}
    b=Q_1^{\sigma_1}\cdots Q_{t}^{\sigma_t},
\end{equation}
where $\sigma_1,\ldots, \sigma_t$ are non-negative integers. Then for any $\epsilon>0$, there exists $C>0$ such that $|\alpha-a/b |\geq Cb^{-1-\epsilon}$
for every $a\in \mathbb{Z}$ and $b$ of the form  \eqref{equation:primefactors}. 
\end{theorem}

\begin{proof}
See \cite{Ridout}.
\end{proof}
Before stating the proof, we note that $f(X) \ll g(X)$ and $f(X) \ll_{a_1,\ldots, a_n} g(X)$  as $f(X)=O(g(X))$ and $f(X)=O_{a_1,\ldots, a_n}(g(X))$ respectively, where $g(X)$ is non-negative. In addition, we write $f(X)\asymp g(X)$ if $f(X)\ll g(X)\ll f(X)$. 
\begin{proof}[Proof of Theorem~\ref{Theorem-fractional-powers} with $k=0$ assuming Theorem~\ref{Theorem-Bernoulli0}] By substituting $d=1/\log b$ and $\theta= \log \alpha / \log b$ in \eqref{eq:thm-Bernoulli0-2} of Theorem~\ref{Theorem-Bernoulli0}, for every $N\geq 2$
\begin{gather*}
\frac{1}{2\pi i}\sum_{1\leq |n|\leq N} \zeta\left(\frac{2\pi in}{\log b}\right) \frac{e^{2\pi i n \log_b \alpha  }}{n}= -  \sum_{0<h<\log_b N} \psi_{1}( b^{h}\alpha  )+G_N +o_{b,\alpha}(\log N ), \\
|G_N| \ll \sum_{0<h< \log _b N} \min \left( 1, \frac{b^h}{\|b^h \alpha \|  N} \right). 
\end{gather*}
Let $\gamma$ be an arbitrarily small positive constant. By substituting $\alpha:=\alpha$, $b:=b^{h}$, and  $\epsilon:=\gamma$ in Theorem~\ref{Theorem:Ridout}, we obtain 
\begin{equation}\label{inequality:Ridout}
    \|b^{h} \alpha\|\gg_\gamma b^{-\gamma h} .
\end{equation}
Therefore, the inequality \eqref{inequality:Ridout} yields that 
\[
\sum_{0< h < \frac{\log_b N}{1+\gamma}  }\min\left(1, \frac{b^{h}}{\|b^h \alpha \| N}\right) 
\ll_\gamma \frac{1}{N}\sum_{0< h < \frac{\log_b N}{1+\gamma}  } b^{(1+\gamma)h} \ll_\gamma 1.
\]
 We also obtain 
\[
\sum_{\frac{\log_b N}{1+\gamma}<h< \log_b N} \min\left(1, \frac{b^{h}}{\|b^h \alpha \| N}\right) 
\leq \sum_{\frac{\log_b N}{1+\gamma}<h< \log_b N} 1
\ll \gamma \log_b N+1, 
\]
where the implicit constant does not depend on $\gamma$. Therefore, we have
\[
G_N=O_\gamma(1) +O(\gamma \log_b N +1  ),
\]
and hence $\limsup_{N\to \infty} G_N/\log N \ll \gamma$, where the implicit constant does not depend on $\gamma$. By choosing $\gamma\to 0$, we conclude $G_N=o(\log N)$. 
\end{proof}

\begin{corollary}\label{Corollary-rational} Let $u$ and $v$ be co-prime integers with $u>v\ge 2$. Let $k\in \mathbb{Z}_{\ge 0}$. We have  
\begin{align*}
&\frac{(k+1)!}{(2\pi i)^{k+1}}\sum_{1\leq |n|\leq N}\zeta\left(-k+\frac{2\pi in}{\log (u/v)}\right) \frac{1} {n^{k+1}}\\
&= -\frac{1}{\log^k(u/v)}  \sum_{0<h<\log_{u/v} N} \psi_{k+1}\left( \left(u/v \right)^{h}\right)+o_{k,u,v}(\log N ) \quad \text{as $N\to \infty$}. 
\end{align*}
Especially, $((u/v )^{h})_{h=1}^\infty$ is uniformly distributed modulo $1$ if and only if for every $k\in \mathbb{Z}_{\ge 0}$
\[
\lim_{N\to\infty} \frac{1}{\log N}\sum_{1\leq |n|\leq N}\zeta\left(-k+\frac{2\pi in}{\log (u/v)}\right) \frac{1} {n^{k+1}}=0.
\]
\end{corollary}

\begin{proof}
Fix any $k\in \mathbb{Z}_{\ge 0}$. By substituting $d=1/\log (u/v)$ and $\theta=0$, we see that $e^{(\theta+h)/d}= (u/v)^h$ for every $h\in \mathbb{Z}_{>0}$. Thus, Theorem~\ref{Theorem-Bernoulli1} leads to Corollary~\ref{Corollary-rational} with $k>0$. 

For $k=0$, take an arbitrarily small constant $\gamma>0$. By Mahler's result \cite[Theorem 2]{Mahler}, each $h\in \mathbb{Z}_{>0}$ satisfies $\|(u/v)^h \| \gg_{u,v,\gamma} e^{-\gamma h}$. In a similar manner with the proof of Theorem~\ref{Theorem-fractional-powers} with $k=0$ assuming Theorem~\ref{Theorem-Bernoulli0}, we obtain Corollary~\ref{Corollary-rational} with $k=0$ from Theorem~\ref{Theorem-Bernoulli0}. 
\end{proof}

Let $k\in \mathbb{Z}_{\ge 0}$. By the Fourier expansion of the $k$-th Bernoulli polynomial \cite[Theorem~4.11]{ArakawaIbukiyamaKaneko},  for every $\theta \in \mathbb{R}$, we have   
\begin{equation}\label{eq:FourierExp}
\lim_{N\to\infty}\frac{(k+1)!}{(2\pi i)^{k+1}  } \sum_{1\leq |n|\leq N }  \frac{e^{2\pi i \theta n}}{ n^{k+1}  } = - \psi_{k+1} (\theta).  
\end{equation}
Thus, it is observed that our main theorems are similar to the quantity obtained from \eqref{eq:FourierExp} by putting $\zeta(-k +2\pi i dn)$ in front of $e^{2\pi i \theta n}/n^{k+1}$. Observing Theorem~\ref{Theorem-Bernoulli0} and Theorem~\ref{Theorem-Bernoulli1}, we see that the above operation causes the transformation from $\psi_{k+1}(\theta)$ to $\sum_{h} \psi_{k+1}(  e^{(h+\theta)/d})$.  We do not fully understand why such a phenomenon occurs.

At last of this section, we give another application of Theorems~\ref{Theorem-Bernoulli0} and \ref{Theorem-Bernoulli1}. 
\begin{corollary}\label{Corollary-SpecialValue}
For all $\alpha\in \mathbb{Z}_{>0}$, $b\in \mathbb{Z}_{\geq 2}$, and  $k\in \mathbb{Z}_{\ge 0}$, we have
\begin{align}\label{eq:main2}
&\lim_{N\to \infty}\frac{1}{\log N} \sum_{1\leq |n|\leq N} \zeta\left(-k+\frac{2\pi in}{\log b}\right) \frac{e^{2\pi i n \log \alpha /\log b  }}{n^{k+1}} \\[4pt] \nonumber
&= 
\left\{
\begin{array}{cl}
-\cfrac{(-1)^{(k+1)/2}(2\pi )^{k+1}B_{k+1}}{(k+1)!\log^{k+1} b} & \text{if $k$ is odd,}\\
\: \\[-9pt]
0 & \text{otherwise}, 
\end{array} \right.
\end{align}
where $B_{k+1}$ denotes the $(k+1)$-st Bernoulli number.   
\end{corollary}

\begin{proof}
Fix arbitrary $\alpha\in \mathbb{Z}_{>0}$ and $b\in \mathbb{Z}_{\ge 2}$.  By substituting $d = 1/\log b$ and $\theta = \log_b \alpha$, we have $e^{(h+\theta)/d}=b^{h} \alpha \in \mathbb{Z}$ for every $h\in \mathbb{Z}_{\ge 0}$.   By \eqref{eq:thm-Bernoulli0-1} in Theorem~\ref{Theorem-Bernoulli0}, 
\[
\frac{1}{2\pi i}\sum_{1\leq |n|\leq N} \zeta\left(\frac{2\pi in}{\log b}\right) \frac{e^{2\pi i n\log_b \alpha }}{n}
=  O((\log N)^{2/3} (\log \log N)^{5/9} ),
\]
which leads to Corollary~\ref{Corollary-SpecialValue} with $k=0$. Take any $k\in \mathbb{Z}_{>0}$. Theorem~\ref{Theorem-Bernoulli1} also yields that   
\[
\frac{(k+1)!}{(2\pi i)^{k+1}}\sum_{1\leq |n|\leq N} \zeta\left(-k+\frac{2\pi in}{\log b}\right) \frac{e^{2\pi i n\log_b \alpha }}{n^{k+1}}= -\frac{1}{\log^k b} \sum_{0<h<\log_b N} \psi_{k+1}(0)+O_{k,b,\alpha}(\log\log N ).
\]
Since $\psi_{k+1}(0)=B_{k+1}$ for every $k\in \mathbb{Z}_{> 0}$ and $B_{k+1}=0$ for every even $k\in \mathbb{Z}_{\ge 0}$, we conclude Corollary~\ref{Corollary-SpecialValue} with $k>0$. 
\end{proof}
By the theory of special values of zeta functions (see \cite[Corollary~4.12]{ArakawaIbukiyamaKaneko}), for every odd number $k\in \mathbb{Z}_{>0}$, we have 
\[
\zeta(k+1)= -\frac{(-1)^{(k+1)/2} (2\pi)^{k+1} B_{k+1} }{2 (k+1)!}.  
\]
Therefore, Corollary~\ref{Corollary-SpecialValue} yields that the left-hand side of \eqref{eq:main2} is equal to 
\[
2\zeta(k+1)/\log^{k+1} b
\]
for every odd $k\in \mathbb{Z}_{>0}$. Further, the case when $k$ is even corresponds to trivial zeros of the Riemann zeta function.

\section{Sketch of the proofs of Theorems~\ref{Theorem-Bernoulli0} and \ref{Theorem-Bernoulli1}}

Let $e(x)=e^{2\pi i x}$ for every $x\in \mathbb{R}$. In this section, we calculate heuristically and describe a sketch of the proof of Theorems~\ref{Theorem-Bernoulli0} and \ref{Theorem-Bernoulli1}.  
For simplicity, we first discuss the case $k\in \mathbb{Z}_{>0}$.  We start with the functional equation of the Riemann zeta function.
\begin{lemma}\label{Lemma:FunctionalEquation} For every $s\in \mathbb{C}\setminus \{1\}$, we have
$\zeta (s) =\chi (s) \zeta(1-s)$, where $\chi(s)=2^{s-1}\pi^s\sec(\pi s/2)/\Gamma(s)$. Further, for any fixed $\sigma\in \mathbb{R}$ and for sufficiently large $t>1$, we have
\[
\chi(\sigma+it)= \left(2\pi/ t \right)^{\sigma+it-1/2}e^{i (t+\pi/4)}\left(1+O\left(\frac{1}{t}\right)\right).
\]
\end{lemma}

By applying Lemma~\ref{Lemma:FunctionalEquation}, we obtain 
\begin{align}\nonumber
\zeta(-k+2\pi idt) &\sim (d t)^{1/2+k-2\pi idt} e(dt+1/8) \zeta (k+1-2\pi idt) \\ \label{equation:strategy1}
&= e(1/8) d^{1/2+k}\sum_{m=1}^\infty \frac{t^{1/2+k}}{m^{k+1}}\cdot  e\left( dt\log \left(\frac{me}{dt}\right) \right),  
\end{align}
where we say that $X\sim Y$ if $X=Y+(\text{errors})$ and the errors are ignorable, but we do not give a precise definition of ``ignorable'' because we will check all the details later. Therefore, we have 
\begin{equation}\label{eq:sketch2}
\sum_{n=1}^N \zeta (-k+2\pi idn) \frac{e(\theta n)}{n^{k+1}} \sim e(1/8) d^{1/2+k} \sum_{m=1}^\infty \frac{1}{m^{k+1}} \sum_{n=1}^N \frac{e(dn\log \left(\frac{me}{dn} \right) +\theta n )}{n^{1/2}}. 
\end{equation}
 For every $x\in [1,N]$, we define 
 \begin{equation}\label{defi:Fm(x)}
 F_m(x)= dx\log (me/(dx))+\theta x.
 \end{equation}
 We need to evaluate
\[
S_m(N)\coloneqq  \sum_{n=1}^N \frac{e(F_m(n))}{n^{1/2}}.
\]
\begin{lemma}\label{Lemma:sum-to-integral}
Let $f(x)$ be a real function with a continuous and steadily decreasing derivative $f'(x)$ in $(a,b)$, and let $f'(b)=\alpha$, $f'(a)=\beta$. Let $g(x)$ be a real positive decreasing function, with a continuous derivative $g'(x)$, and let $|g'(x)|$ be steadily decreasing. Then
\begin{align*}
\sum_{a<n \leq b} g(n) e(f(n)) &= \sum_{\alpha -\eta < h < \beta +\eta} \int_{a}^b g(x) e(f(x)-h x) dx \\
&+O (g(a)\log (\beta -\alpha +2) ) +O(|g'(a)|),    
\end{align*}
where $\eta$ is any positive constant less than $1$.  
\end{lemma}
\begin{proof}
See  \cite[Lemma~4.10]{Titchmarsh}.
\end{proof}

We ignore all the errors and the difference between $\sum_{a<n\leq b}$ and $\sum_{a\leq n\leq b}$ in Lemma~\ref{Lemma:sum-to-integral} for the moment. We now choose $\eta=0$ as in Lemma~\ref{Lemma:sum-to-integral}. It is wrong since $\eta$ should be positive, but we do not pay attention here. By simple calculations, we have 
\begin{equation}\label{eq:derivF}
F'_m(u)= d \log(m/(du))+\theta.
\end{equation}
Thus,  Lemma~\ref{Lemma:sum-to-integral} with $f(x)\coloneqq F_m(x)$, $g(x)\coloneqq x^{-1/2}$, $a\coloneqq 1$, and $b\coloneqq N$ implies that 
\begin{align}\label{equation:S_m-approx}
S_m(N)\sim   \sum_{\alpha <h \leq \beta} \int_{1}^N  \frac{e(F_m(x) -hx )}{x^{1/2}} dx,
\end{align}
where $\beta=\beta(m)=F_m'(1)$, $\alpha=\alpha(m)=F_m'(N)$. Let $\xi= \xi_{m,h}$ be a real number satisfying $F_m'(\xi)=h$. Then \eqref{eq:derivF} leads to
\[
\xi=\xi_{m,h} = \frac{m}{d}e^{(\theta-h)/d}.
\]
We note that $F_m'(\xi)-h=0$ by choice of $\xi$. If $\xi_{m,h}\in (1,N)$, then by applying the stationary phase integral \cite[(3)]{Huxley}, we have
\begin{equation}\label{eq:SPI1}
 \int_{1}^N  \frac{e(F_m(x) -hx )}{x^{1/2}} dx\sim \frac{e(F_m(\xi) -h\xi-1/8)}{\xi^{1/2} |F_m''(\xi)|^{1/2}  }
\end{equation}

By \eqref{eq:derivF},  $F'_m(\xi)= h$, and $F_m(u)=uF'_m(u)+du$, we have
\begin{gather*}
 F_m(\xi)-h \xi=  d \xi,\quad F''_m(\xi)= -d/\xi,\\
\beta(m)= d\log(m/d)+\theta,\quad \alpha(m)=  d\log(m/(dN))+\theta.
\end{gather*}
Therefore, by combining \eqref{equation:S_m-approx} and \eqref{eq:SPI1}, we obtain
\begin{equation}\label{eq:sketch5}
S_m(N)\sim   \sum_{\alpha(m) <h \leq \beta(m)} \frac{e(-1/8)e(m e^{(\theta-h)/d})}{d^{1/2}  }.
\end{equation}
We note that \eqref{eq:SPI1} is available only in the case $\xi_{m,h}\in (1,N)$, but we now skip to care the cases $\xi\leq 1$ or $\xi\geq N$, and we suppose that \eqref{eq:SPI1} is always true.  Combining \eqref{eq:sketch2} and \eqref{eq:sketch5}, we obtain 
\begin{align*}\label{equation:strategy2}
\sum_{n=1}^N \zeta (-k+2\pi idn) \frac{e(\theta n)}{n^{k+1}} &\sim d^k \sum_{m=1}^\infty \frac{1}{m^{k+1}} \sum_{\alpha(m) < h\leq \beta(m)} e\left(m e^{(\theta-h)/d} \right)\\ 
&=d^k \sum_{m=1}^\infty \frac{1}{m^{k+1}} \left(\sum_{\alpha(m) < h<0}+\sum_{0 \leq h\leq \beta(m)} \right)e\left(m e^{(\theta-h)/d} \right).
\end{align*}
Since $\beta (m)\ll_{d, \theta} \log (m+1)$ and $k\geq 1$, we have 
\[
d^k \sum_{m=1}^\infty \frac{1}{m^{k+1}} \sum_{0\leq h\leq \beta (m)} e(m e^{(\theta-h)/d} ) \ll_{d,k} \sum_{m=1}^\infty \frac{\log (m+1)}{m^{k+1}}  \ll_k 1.  
\]
Moreover, by interchanging the double summation over $m$ and $h$, 
\[
\sum_{n=1}^N \zeta (-k+2\pi idn) \frac{e(\theta n)}{n^{k+1}} \sim d^k \sum_{\alpha(1) < h<0 } \sum_{\substack{ 1\leq m \leq dNe^{(h-\theta)/d}  } }  \frac{e(me^{(\theta-h)/d})}{m^{k+1}}.
\]
 By the Schwarz reflection principle, $\zeta(\overline{s})=\overline{\zeta(s)}$ holds, and hence 
\[
\sum_{1\leq |n|\leq N} \zeta (-k+2\pi idn) \frac{e(\theta n)}{n^{k+1}} \sim d^k \sum_{\alpha(1) \leq h<0 } \sum_{\substack{ 1\leq |m| \leq dNe^{(h-\theta)/d}  } }  \frac{e(me^{(\theta-h)/d})}{m^{k+1}}.
\]
The Fourier expansion of $\psi_{k+1}$ (see \cite[Theorem~4.11]{ArakawaIbukiyamaKaneko}) establishes 
\[
\sum_{1\leq|m|\leq dNe^{(h-\theta)/d} } \frac{e(me^{(\theta-h)/d})}{m^{k+1}} \sim  -\frac{(2\pi i)^{k+1}}{(k+1)!} \psi_{k+1}(e^{(\theta-h)/d}).
\]
Therefore, by replacing $h$ with $-h$, we conclude that
\[
\frac{(k+1)!}{(2\pi i)^{k+1}}\sum_{1\leq |n|\leq N} \zeta (-k+2\pi idn) \frac{e(\theta n)}{n^{k+1}} \sim -d^k \sum_{0<h<d\log N} \psi_{k+1}( e^{(\theta+ h)/d}).
\]
In the case $k=0$, the above discussion gets much more difficult because the series in \eqref{equation:strategy1} does not converge absolutely. Thus, we will do the above discussion much more carefully.

We organize the remainder of the article as follows. In Section~\ref{section:Applying_the_functional_equation}, we apply the functional equation for the Riemann zeta function. In Section~\ref{section:Applying_the_Euler-Maclaurin_formula}, we translate from a double summation to a superposition of a single summation and a single integration using the Euler-Maclaurin formula. In Section~\ref{section:Restricting_the_ranges_of_summation_and_integration_I} and \ref{section:Restricting_the_ranges_of_summation_and_integration_II}, we restrict the ranges of the summation and integration. In Section~\ref{section:Applying_the_stationary_phase_method}, we calculate the integration by applying the stationary phase method. At last, we complete the proof of Theorems~\ref{Theorem-Bernoulli0} and \ref{Theorem-Bernoulli1} in Section~\ref{section:Proof_of_Theorems_2.1_and_2.2}.

\section{Applying the functional equation}\label{section:Applying_the_functional_equation}
Let $\theta$,  $\sigma$, and $d>0$ be fixed real numbers with $\sigma\leq 0$ and $d>0$. We consider the parameters $\theta$, $\sigma$, and $d$ as constants. Thus, we omit the dependencies of these parameters. In this section, we aim to give a proof of the following proposition which is a precise form of \eqref{eq:sketch2}.

\begin{proposition}\label{Proposition:MeanValueStep1}
There exists $C>1$ such that for all $M,N\in \mathbb{N}$ with $M \geq CN$, 
\[
\sum_{n=1}^N \zeta(\sigma+2\pi idn)\cdot \frac{e(\theta n)}{n^{1-\sigma}} = D \sum_{m=1}^M \frac{1}{m^{1-\sigma}} \sum_{n=1}^N \frac{e(F_m(n)) }{n^{1/2}} + O(1),   
\]
where  $D= e(1/8) d^{1/2-\sigma}$ and $F_m(x)$ is defined in \eqref{defi:Fm(x)}.
\end{proposition}

\begin{lemma}\label{Lemma:zeta-critical} We have
\[
\zeta (1-\sigma+it) =\sum_{m=1}^M \frac{1}{m^{1-\sigma+it}}  +O(M^\sigma |t|^{-1} +M^{-1+\sigma} )  
\]
uniformly for $M\in \mathbb{N}$ and $|t|< 2\pi M/C$, where $C$ is a given constant greater than $1$. 
\end{lemma}

\begin{proof} See \cite[Theorem~4.11]{Titchmarsh}.  
\end{proof}

\begin{remark}
Lemma~\ref{Lemma:zeta-critical} is trivial if $\sigma<0$. Thus, we need the lemma only for the case $\sigma=0$ which corresponds to Theorem~\ref{Theorem-Bernoulli0}.  
\end{remark}

\begin{lemma}\label{Lemma:zeta(1+it)}
    For $|t|\geq1$, we have
    \begin{align*}
        \zeta(1-\sigma+it)\ll
        \begin{cases}
            \log^{2/3} (|t|+1)  \quad & \text{if } \sigma=0,\\
            1  \quad & \text{if } \sigma<0.
        \end{cases}
    \end{align*}
\end{lemma}
\begin{proof}
See \cite[Chapter IV, Section 2, Corollary 1]{Karatsuba-Voronin}.
\end{proof}

\begin{comment}
By choosing $M\asymp |t|$ and $\sigma=0$ in Lemma~\ref{Lemma:zeta-critical}, we obtain $\zeta(1+it)\ll \log(|t|+1)$. Therefore, for every $|t|\geq 1$, we have 
\begin{equation}\label{Inequality:zeta(1+it)}
\zeta (1-\sigma +it) \ll 
\left\{\begin{array}{ll} 
\log (|t|+1) \quad &\text{if } \sigma=0, \\
1\quad &\text{otherwise.}
\end{array}\right.
\end{equation}
\end{comment}

\begin{proof}[Proof of Proposition~\ref{Proposition:MeanValueStep1}] 
Applying Lemma~\ref{Lemma:FunctionalEquation} (the functional equation of the Riemann zeta function) and Lemma~\ref{Lemma:zeta(1+it)}, we obtain 
\begin{align*}
\zeta(\sigma+2\pi idt) &= (d t)^{1/2-\sigma-2\pi idt} e(dt+1/8) \zeta (1-\sigma-2\pi idt) \left(1+O\left(\frac{1}{|t|}\right)\right)\\
&= D t^{1/2-\sigma} e\left(dt\log\left(\frac{e}{dt}\right)   \right) \zeta (1-\sigma-2\pi idt) +O(|t|^{-1/2-\sigma}\log^{2/3} (|t|+1)),
\end{align*}
where $D=d^{1/2-\sigma}e(1/8)$. Therefore, if $M\geq C d N+1$, then Lemma~\ref{Lemma:zeta-critical} implies that 
\begin{align*}
\zeta(\sigma+2\pi idt) &=D t^{1/2-\sigma} e\left(dt\log\left(\frac{e}{dt}\right)   \right) \sum_{m=1}^M \frac{1}{m^{1-\sigma-2\pi id t }} \\
&\quad+O(|t|^{1/2-\sigma} M^{-1+\sigma} +  |t|^{-1/2-\sigma}\log^{2/3} (|t|+1))
\end{align*}
uniformly for $|t|\leq 2\pi dN$. 
Therefore, we obtain 
\begin{equation}\label{equation:FE2}
\sum_{n=1}^N \zeta(\sigma+2\pi idn)\cdot \frac{e(\theta n)}{n^{1-\sigma}} = D \sum_{m=1}^M \frac{1}{m^{1-\sigma}}\sum_{n=1}^N \frac{e(F_m(n))}{n^{1/2}} + O(N^{1/2}M^{-1+\sigma}+ 1 ), 
\end{equation}
By \eqref{equation:FE2} and $M\gg N$, we conclude Proposition~\ref{Proposition:MeanValueStep1}.  
\end{proof}

We now choose $M\in \mathbb{N}$ as a parameter satisfying that $C_1N\leq  M \leq C_2 N$, where $C_1$ and $C_2$ are sufficiently large positive constants depending only on $\theta$, $\sigma$, and $d$.

\section{Applying the Euler-Maclaurin formula}\label{section:Applying_the_Euler-Maclaurin_formula}

In this section, we aim to show the following proposition. 

\begin{proposition}\label{Proposition:MeanValueStep2}
For all $m,M,N\in \mathbb{N}$ with $m\leq M$ and $C_1 N\leq M\leq C_2 N$, we have 
\begin{align*}
\sum_{n=1}^N \frac{e(F_m(n) )}{n^{1/2}}
=  \int_{1}^N \frac{e(F_m(u))}{u^{1/2}} du 
+  \sum_{h\neq 0} \frac{1}{h} \int_{1}^N \frac{F_m'(u)}{u^{1/2}}  e(F_m(u)-hu) du + R_m,
\end{align*}
where we define $\sum_{h\neq 0} = \lim_{H\to \infty} \sum_{1\leq |h|\leq H}$. Further, $R_m$ satisfies $\sum_{m=1}^M R_m /m^{1-\sigma} \ll 1$. 
\end{proposition}
Let $\psi(u)=\psi_1(u)$ for all $u\in \mathbb{R}$. 
\begin{lemma}[the Euler-Maclaurin formula]\label{Lemma:EM-formula} For all $N\in \mathbb{N}$ and for all continuously differentiable functions $f$ defined on $[1,N]$, we have 
\[
\sum_{n=1}^N f(n) = \int_{1}^N f(u) du + \int_{1}^N \psi(u) f'(u) du + \frac{1}{2} (f(N)+f(1)). 
\]
\end{lemma}
\begin{proof}
See \cite[Theorem~5.1]{ArakawaIbukiyamaKaneko}. 
\end{proof}

\begin{lemma}\label{Lemma:e-zeta}
For all $M,N\in \mathbb{N}$ with $C_1 N\leq M\leq C_2 N$ and $u\in [1,N]$, we have 
\[
\sum_{m=1}^M \frac{e(F_m(u)) }{m^{1-\sigma}} \ll_\sigma
\left\{
\begin{array}{ll}
 \log^{2/3} (u+1) &\text{if } \sigma=0, \\
 1 	&\text{otherwise}.
\end{array}\right.
\]
\end{lemma}
\begin{proof}
It suffices to show the case when $\sigma=0$ because $\sigma$ is not positive. By the definition of $F_m(u)$, we observe that 
\[
\left|\sum_{m=1}^M \frac{e(F_m(u)) }{m}\right|= \left| \sum_{m=1}^M \frac{e(d u \log m ) }{m}\right| = \left| \sum_{m=1}^M \frac{1}{m^{1+2\pi idu}}\right|.     
\]
Lemmas~\ref{Lemma:zeta-critical} and \ref{Lemma:zeta(1+it)} imply that 
\[
\sum_{m=1}^M \frac{1}{m^{1+2\pi idu}}=  \zeta(1+2\pi i du) + O(+u^{-1}+M^{-1} ) \ll \log^{2/3} (u+1) .  
\]
\end{proof}

\begin{proof}[Proof of Proposition~\ref{Proposition:MeanValueStep2}] 
Lemma~\ref{Lemma:EM-formula} with $f(n)\coloneqq e(F_m(n))/n^{1/2}$ leads to 
\begin{align*}
\sum_{n=1}^N \frac{e(F_m(n) )}{n^{1/2}}&= \int_{1}^N \frac{e(F_m(u))}{u^{1/2}}  du + \int_{1}^N \psi(u) \frac{d}{du} \left( \frac{e(F_m(u))}{u^{1/2}} \right) du \\
&\quad\quad + \frac{1}{2} \left( e(F_m(1)) + \frac{e(F_m(N)) }{N^{1/2}}  \right).  
\end{align*}
The second integral is equal to 
\[
\int_1^N \psi (u) \left(-\frac{1}{2} \frac{e(F_{m}(u) )}{u^{3/2}} + 2\pi i \frac{F'_m(u)}{u^{1/2}}  e(F_m(u)) \right) du. 
\]
By setting 
\[
R_m=- \frac{1}{2} \int_{1}^N \frac{\psi (u)}{u^{3/2}} e(F_m(u)) du   +\frac{1}{2} \left( e(F_m(1)) + \frac{e(F_m(N)) }{N^{1/2}}  \right),    
\]
we obtain 
\begin{equation}\label{Equation:Sum-e(Fm(n))}
\sum_{n=1}^N \frac{e(F_m(n) )}{n^{1/2}}=  \int_{1}^N \frac{e(F_m(u))}{u^{1/2}}  du +  \int_1^N \psi (u)\cdot 2\pi i \frac{F'_m(u)}{u^{1/2}}  e(F_m(u)) du+R_m. 
\end{equation}
Since $\psi(u)= \sum_{h\neq 0} e(-hu)/(2\pi i h)$ and the Fourier series converges boundedly, by applying the bounded convergence theorem, the second integral on the right-hand side of \eqref{Equation:Sum-e(Fm(n))} is equal to
\[
\sum_{h\neq 0} \frac{1}{h} \int_{1}^N \frac{F_m'(u)}{u^{1/2}}  e(F_m(u)-hu) du
\]
Moreover, if $\sigma=0$, then Lemma~\ref{Lemma:e-zeta} implies that 
\begin{align*}
\sum_{m=1}^M \frac{R_m}{m^{1-\sigma}} 
&=-\frac{1}{2}\int_1^N \frac{\psi(u)}{u^{3/2}}\sum_{m=1}^M \frac{e(F_m(u))}{m} du + \frac{1}{2}\sum_{m=1}^M \frac{e(F_m(1))}{m} + \frac{1}{2N^{1/2}} \sum_{m=1}^M \frac{e(F_m(N))}{m}\\
&\ll \int_{1}^N \frac{\log^{2/3} (u+1)}{u^{3/2}} du + 1 + \frac{\log^{2/3} (N+1)}{N^{1/2}} \ll 1.
\end{align*}
If $\sigma<0$, then we also have
\begin{align*}
\sum_{m=1}^M \frac{R_m}{m^{1-\sigma}} 
&\ll \int_{1}^N \frac{1}{u^{3/2}} du  + \frac{1}{N^{1/2}}  \ll 1,
\end{align*}
which completes the proof of Proposition~\ref{Proposition:MeanValueStep2}.
\end{proof}

\section{Restricting the ranges of summation and integration I}\label{section:Restricting_the_ranges_of_summation_and_integration_I}
By combining Propositions~\ref{Proposition:MeanValueStep1} and \ref{Proposition:MeanValueStep2}, for all $M, N\in \mathbb{Z}_{\ge 1}$ with $C_1N\leq M \leq C_2 N$, 
\begin{align}\label{equation:sum_integral-Euler_Maclaurin}
\sum_{n=1}^N \zeta(\sigma+2\pi idn) \frac{e(\theta n)}{n^{1-\sigma}} &= D \sum_{m=1}^M \frac{1}{m^{1-\sigma}} \int_{1}^N \frac{e(F_m(u))}{u^{1/2}} du \\
&\quad +  D \sum_{m=1}^M \frac{1}{m^{1-\sigma}} \sum_{h\neq 0} \frac{1}{h} \int_{1}^N \frac{F_m'(u)}{u^{1/2}}  e(F_m(u)-hu) du   + O(1),\notag
\end{align}
where $C_1$ and $C_2$ are some fixed constants with $C_2>C_1>1$, $D=e(1/8) d^{1/2-\sigma}$ and $F_m(u)=du\log (me/(du))+\theta u$. In this section, we aim to show Proposition~\ref{proposition:zeta-restrictH}.

\begin{lemma}[{\cite[Lemma~4.2]{Titchmarsh}}]\label{lemma:exponential-int-first}
Let $f$ and $g$ be real functions defined on $[a,b]$. Assume that $f$ and $g$ are twice differentiable on $[a,b]$, and $g/f'$ is monotonic throughout the interval $[a,b]$. Suppose that there exists $M>0$ such that for every $x\in [a,b]$, we have $|f'(x)/g(x)|\geq M$. Then 
\[
\left|\int_a^b g(x)e(f(x))dx\right| \ll \frac{1}{M}.
\]
\end{lemma}

\begin{lemma}[{\cite[Lemma~4.5]{Titchmarsh}}]\label{lemma:second-derivative}
Let $f(x)$ be a real function, twice differentiable, and let $f''(x)\geq r>0$ or $f''(x)\leq r<0$, throughout the interval $[a,b]$. Suppose that $g(x)/f'(x)$ is monotonic and $|g(x)|\leq M$. Then 
\[
\int_a^b g(x) e(f(x)) \ll M r^{-1/2}.
\]
\end{lemma}

Let $\alpha$ be a fixed constant in $(0,1]$, and let $c=c(d,\theta)$ be a sufficiently large parameter. We set $H= c \log^{1+\alpha}(MN)$.

\begin{lemma}\label{Lemma:restrict-h}
 For all $N, M\in \mathbb{Z}_{\ge 2}$ with $M \asymp N$, we have
\[
\sum_{m=1}^M \frac{1}{m^{1-\sigma}} \sum_{|h|> H } \frac{1}{h} \int_{1}^N \frac{F_m'(u)}{u^{1/2}}  e(F_m(u)-hu) du \ll_{\alpha,\sigma} \begin{cases}
\log^{1-\alpha}N & \text{if $\sigma=0$}; \\
1  & \text{if $\sigma<0$}.
\end{cases}
\]
\end{lemma}

\begin{proof}
 We observe that $|F'_m(u)|= |d\log(m/(du)) +\theta |\leq d\log (MN)+O_{\theta,d}(1)$. Therefore, there exists a constant $c=c(\theta,d)>0$ such that $|F'_m(u)|\leq |h|/2$ for all $|h|\geq H=c\log^{1+\alpha} (MN)$, $u\in[1,N]$, and $m\in\mathbb{Z}_{[1,M]}$. Assume that $h\geq H$. To apply Lemma~\ref{lemma:exponential-int-first}, we define
\[
a_h(u)= \frac{1}{u^{1/2}} \cdot \frac{F'_m(u)}{h-F'_m(u)}
\]
for all $u\in [1,N]$ and $m\in \mathbb{N}$. By $|F'_m(u)|\leq |h|/2$, the sign of $h-F'_m(u)$ is always positive. Further, since $F_m''(u)=-d/u$, we see that 
\[
\frac{\partial}{\partial u} \left(\frac{F'_m(u)}{h-F'_m(u)}\right)=\frac{F''_m(u) h}{(h-F'_m(u))^2}= -\frac{d h}{u(h-F'_m(u))^2}<0.
\]
Therefore, for every fixed $h\geq H$, $a_h(u)$ is monotonically decreasing on  $u\in [1,N]$, and hence we have $|a_h(u)|\leq |a_h(1)|+ |a_h(N)|$ on $u\in [1,N]$. By applying Lemma \ref{lemma:exponential-int-first} with $g(u)\coloneqq F'_m(u)u^{-1/2}$, $f(u)\coloneqq hu-F_m(u)$, and $M^{-1}=|a_h(1)|+ |a_h(N)|$, we have 
\begin{align}\label{Inequality:restrict-h-1}
&\sum_{h\ge H } \frac{1}{h} \int_{1}^N \frac{F_m'(u)}{u^{1/2}}  e(F_m(u)-hu) du  \\
&\ll \sum_{h\ge H} \frac{|a_h(1)|+ |a_h(N)|}{h}   \ll (|F_m'(1)|+|F'_m(N)|) \sum_{h\ge H}  \frac{1}{h^2}\ll \frac{1} {\log^{\alpha}N},\nonumber
\end{align}
where the last inequality follows from $|F_m'(1)|+|F_m'(N)|\ll\log(MN)$ and $H=c\log^{1+\alpha}(MN)$. In the case $h\leq -H$, we also obtain the same upper bound as the most right-hand side of \eqref{Inequality:restrict-h-1}, and hence
\begin{align*}
&\sum_{m=1}^M \frac{1}{m^{1-\sigma}} \sum_{|h|\ge H } \frac{1}{h} \int_{1}^N \frac{F_m'(u)}{u^{1/2}}  e(F_m(u)-hu) du\\
&\ll  \frac{1}{\log^{\alpha} N}\sum_{m=1}^M \frac{1}{m^{1-\sigma}}  \ll \begin{cases}
\log^{1-\alpha}N & \text{if $\sigma=0$}; \\
1  & \text{if $\sigma<0$}.
\end{cases}
\end{align*} 
\end{proof}

\begin{lemma}\label{Lemma:saw-tooth-function}
For every $K\in \mathbb{N}$, we have 
\[
\left|\sum_{k=1}^K \frac{\sin (2\pi kx)}{\pi k}+\psi(x)\right| \leq \min \left(\frac{1}{2}, \frac{1}{(2K+1)\pi |\sin \pi x|}\right).
\]
\end{lemma} 

\begin{proof}
See \cite[Lemma~D.1]{Montgomery-Vaughan}. 
\end{proof}

\begin{lemma}\label{Lemma:delete-h} For all $M, N\in \mathbb{Z}_{\ge 1}$ with $C_1N\leq M \leq C_2 N$ and $C_2>C_1>1$, we have
\begin{align*}
\sum_{n=1}^N \zeta(\sigma+2\pi idn) \frac{e(\theta n)}{n^{1-\sigma}}&=D \sum_{m=1}^M \frac{1}{m^{1-\sigma}}  \sum_{|h|\leq H} \int_{1}^N \frac{e(F_m(u)-hu)}{u^{1/2}} du \\
&+O(1 + \ichi(\sigma=0)  \log^{1-\alpha}N ),
\end{align*}
where $\ichi(P)=1$ if $P$ is true;   $\ichi(P)=0$ otherwise for every statement $P$. 
\end{lemma}
\begin{proof}
By \eqref{equation:sum_integral-Euler_Maclaurin} and Lemma~\ref{Lemma:restrict-h}, we obtain 
\begin{align*}
\sum_{n=1}^N \zeta(\sigma+2\pi idn) \frac{e(\theta n)}{n^{1-\sigma}} &= D \sum_{m=1}^M \frac{1}{m^{1-\sigma}} \int_{1}^N \frac{e(F_m(u))}{u^{1/2}} du \\
&\quad +  D \sum_{m=1}^M \frac{1}{m^{1-\sigma}} \sum_{0<|h|\leq H} \frac{1}{h} \int_{1}^N \frac{F_m'(u)}{u^{1/2}}  e(F_m(u)-hu) du  \\
& \quad + O(1+\ichi(\sigma=0) \log^{1-\alpha} N ).
\end{align*}

By integration by parts, for all $1\leq m\leq M$ and $0<|h|\leq H$, we obtain 
\begin{align*}
&\frac{1}{h} \int_{1}^N \frac{F_m'(u)}{u^{1/2}}  e(F_m(u)-hu) du\\
&= \frac{1}{2\pi i h}\int_{1}^N \frac{1}{u^{1/2}} e(-hu) \frac{d}{du} (e(F_m(u))) du  \\
&=\frac{1}{2\pi i h} \left(\frac{e(F_m(N)-hN)}{N^{1/2}} -e(F_m(1)-h) \right)\\
&\quad + \frac{1}{2\pi i h} \int_{1}^N \left( \frac{ e(F_m(u)-hu)}{2u^{3/2}}  + \frac{2\pi i h e(F_m(u)-hu )}{u^{1/2}}\right)\ du .
\end{align*}
In addition, since $\sum_{0<|h|\leq H} h^{-1}=0$ , we observe that 
\begin{gather*}
\sum_{m=1}^M \frac{1}{m^{1-\sigma}} \sum_{0<|h|\leq H} \frac{e(F_m(1)-h)}{2\pi i h} = \sum_{m=1}^M \frac{e(F_m(1))}{m^{1-\sigma}} \sum_{0<|h|\leq H} \frac{1}{2\pi i h} =0,\\ 
\sum_{m=1}^M \frac{1}{m^{1-\sigma}} \sum_{0<|h|\leq H} \frac{e(F_m(N)-hN)}{2\pi i h} N^{-1/2}=N^{-1/2}\sum_{m=1}^M \frac{e(F_m(N))}{m^{1-\sigma}} \sum_{0<|h|\leq H} \frac{1}{2\pi i h} =0.
\end{gather*}
Therefore, Lemma~\ref{Lemma:e-zeta} and Lemma~\ref{Lemma:saw-tooth-function} imply that
\begin{align*}
&\sum_{m=1}^M \frac{1}{m^{1-\sigma}} \sum_{0<|h|\leq H} \frac{1}{4\pi ih} \int_{1}^N \frac{e(F_m(u)- hu )}{u^{3/2}} du \\
&\ll \int_{1}^N \frac{1}{u^{3/2}} \left| \sum_{m=1}^M \frac{e(F_m(u))}{m^{1-\sigma}} \right|\left|\sum_{0<|h|\leq H} \frac{e(-hu)}{2\pi ih} \right| du \\
&\ll \int_{1}^N \frac{\log^{2/3} (u+1)}{u^{3/2}} du \ll 1, 
\end{align*}
which completes the proof of Lemma~\ref{Lemma:delete-h}.
\end{proof}

\begin{lemma}\label{Lemma:mean_value_of_sum_of_e}
    Let $a$ and $b$ be integers with $a<b$. We have
    \[
    \int_a^b \left| \sum_{|h|\leq H} e(-hu) \right| du \ll (b-a)\log H.
    \]
\end{lemma}
\begin{proof}
By the periodicity of $e(\cdot)$, we have
\begin{align*}
    \frac{1}{b-a} \int_a^b \left| \sum_{|h|\leq H} e(-hu) \right| du
    &=\int_0^1 \left| \sum_{|h|\leq H} e(-hu) \right| du
    \ll \int_0^1 \min(H, \|u\|^{-1}) du\\
    &=\left( \int_0^{H^{-1}}+\int_{H^{-1}}^{1-H^{-1}}+\int_{1-H^{-1}}^1 \right) \min(H, \|u\|^{-1}) du\\
    &=H\left( \int_0^{H^{-1}}+\int_{1-H^{-1}}^1 \right) du + 2\int_{H^{-1}}^{1/2} u^{-1} du\\
    &\ll \log H.
\end{align*}
\end{proof}

\begin{comment}
\begin{lemma}\label{Lemma:Euler_Maclaurin_for_decreasing_functions}
    Let $(a_i)_{0\leq i\leq n}$ be a monotone increasing sequence. Let $f$ be a monotone decreasing function on $[a_0, a_n]$. Then,
    \[
    \sum_{i=1}^n (a_i-a_{i-1}) f(a_i) \leq \int_{a_0}^{a_n} f(x) dx.
    \]
\end{lemma}
\begin{proof}
    \textcolor{red}{文献を挙げる. Hardy, Littlewood, Polyaの本Inequalityなどに載っている.}
\end{proof}

\begin{lemma}\label{Lemma:Abels_identity}
    For any arithmetical function $a(n)$ let $A(x)=\sum_{n\leq x} a(n)$. Assume $f$ has a continuous derivative on the interval $[1,x]$. Then we have
    \[
    \sum_{n\leq x} a(n)f(n) = A(x)f(x)-\int_1^x A(t) f'(t) dt.
    \]
\end{lemma}
\begin{proof}
    See \cite[Theorem 4.2]{Apostol}.
\end{proof}
\end{comment}

\begin{lemma}\label{Lemma:restrictX}
For all $X\in \mathbb{Z}$ with $2\leq X\leq \min (M,N)$, we have
\begin{equation}\label{Equation:smallu}
\sum_{m=1}^M \frac{1}{m^{1-\sigma}} \sum_{|h|\leq H} \int_{1}^X \frac{ e(F_m(u)-hu)}{u^{1/2}} du\ll \begin{cases}
X^{1/2} (\log X)^{2/3} \log H & \text{if $\sigma=0$};\\
X^{1/2} \log H & \text{if $\sigma<0$.}
\end{cases}
\end{equation}
\end{lemma}
\begin{proof}
Let $\Lambda$ be the left-hand side of \eqref{Equation:smallu}, that is, 
\[
\Lambda= \int_{1}^X \left( \frac{1}{u^{1/2}}\sum_{m=1}^M \frac{e(F_m(u))}{m^{1-\sigma}}\right) \left( \sum_{|h|\leq H} e(-hu)\right) du. 
\]
In the case $\sigma<0$, by Lemma~\ref{Lemma:e-zeta}, we obtain 
\begin{align*}
|\Lambda| &\leq \sup_{u\in [1,X] } \left| \sum_{m=1}^M \frac{e(F_m(u))}{m^{1-\sigma}}\right| \int_{1}^X  \frac{1}{u^{1/2}}  \left|   \sum_{|h|\leq H} e(-hu)  \right| du \\
&\ll \sum_{r=1}^R \int_{2^{r-1}}^{\min (2^{r}, X) }    \frac{1}{u^{1/2}}   \left|   \sum_{|h|\leq H} e(-hu)  \right| du,
\end{align*}
where $R$ is the integer satisfying $2^{R}\leq X < 2^{R+1}$. By Lemma~\ref{Lemma:mean_value_of_sum_of_e},  we have  
\[
\sum_{r=1}^R \int_{2^{r-1}}^{\max (2^{r}, X) }    \frac{1}{u^{1/2}}   \left|   \sum_{|h|\leq H} e(-hu)  \right| du \ll \sum_{r=1}^R 2^{-r/2} 2^r \log H  \ll X^{1/2} \log H. 
\]
Suppose that $\sigma=0$. Lemma~\ref{Lemma:e-zeta} leads to
\begin{align*}
|\Lambda| &\leq \sup_{u\in [1,X] } \left| \sum_{m=1}^M \frac{e(F_m(u))}{m}\right| \int_{1}^X  \frac{1}{u^{1/2}}  \left|   \sum_{|h|\leq H} e(-hu)  \right| du \\
& \ll  (\log X)^{2/3}   \int_{1}^X  \frac{1}{u^{1/2}}  \left|   \sum_{|h|\leq H} e(-hu)  \right| du.  
\end{align*}
By the proof of Lemma~\ref{Lemma:restrictX} for $\sigma<0$, we recall that 
\[
\int_{1}^X  \frac{1}{u^{1/2}}  \left|   \sum_{|h|\leq H} e(-hu)  \right| du \ll X^{1/2} \log H,
\]
and hence $\Lambda\ll X^{1/2}(\log X)^{2/3} \log H$.  
\end{proof}

Let $X=X(\sigma)$ be a parameter of integers with $2\leq X\leq \min (M,N)$. By combining Lemmas~\ref{Lemma:delete-h} and \ref{Lemma:restrictX}, we obtain the following proposition.

\begin{proposition}\label{proposition:zeta-restrictH}
For all $N, M\in \mathbb{Z}_{\geq2}$ with $C_1N\leq M \leq C_2 N$ and $C_2>C_1>1$, we have
\begin{align} \label{eq:zeta-restrictH}
&\sum_{n=1}^N \zeta(\sigma+2\pi idn) \frac{e(\theta n)}{n^{1-\sigma}}\\ \nonumber
&=D \sum_{m=1}^M \frac{1}{m^{1-\sigma}}  \sum_{|h|\leq H} \int_{X}^N \frac{1}{u^{1/2}} \cdot  e(F_m(u)-hu) du \\ \nonumber
&+O(1+ X^{1/2} \log H + \ichi(\sigma=0) (\log^{1-\alpha}N + X^{1/2} (\log X)^{2/3} \log H ) ).
\end{align}
\end{proposition}

\section{Restricting the ranges of summation and integration I\hspace{-2pt}I}\label{section:Restricting_the_ranges_of_summation_and_integration_II}

For every $m\in \mathbb{Z}_{[1,M]}$, $h\in\mathbb{Z}_{[-H,H]}$, and real number $u\in [X,N]$, we set 
\[
f(u)=f_{m,h}(u)= F_m(u)-hu=d u \log\left(\frac{e m}{d u}\right)+\theta u-hu.
\]
Then we have
\[
f'(u) = d \log \left( \frac{m}{d u} \right)+\theta-h =0 \iff u= \frac{m}{d} e^{(\theta-h)/d} \eqqcolon \xi_{m,h}   
\]
Through the stationary phase integration in \eqref{eq:SPI1}, if $\xi_{m,h}\in [X,N]$, then the integral $\int_{X}^N u^{-1/2} e(f(u)) du$  should be large. We here observe that 
\begin{align*}
X \leq \xi_{m,h} \leq N &\iff \frac{d X }{m} \leq e^{(\theta-h)/d}  \leq  \frac{d N }{m} \\
& \iff d \log\left(\frac{m}{d N} \right)+\theta \leq h\leq   d \log\left(\frac{m}{d X} \right)+\theta.  
\end{align*}
Let $H_1=H_1(m)= d \log\left(\frac{m}{d N} \right)+\theta$ and let $H_2=H_2(m)=d \log\left(\frac{m}{d X} \right)+\theta$. Note that 
\[
H_2-H_1=d\log(N/X). 
\]
In addition, let $V$ be a parameter in $(0,1/4]$ determined later. For every $m\in \mathbb{Z}_{[1,M]}$, let 
\[
H_1'=H_1'(m)=d\log \left(\frac{m}{d(1+V)N}\right) +\theta, \quad H_1''=H_1''(m)=d\log \left(\frac{m}{d(1-V)N}\right) +\theta.
\]
In Section \ref{section:Restricting_the_ranges_of_summation_and_integration_II}, we divide the summation regarding $h$ into the following six parts:
\begin{align*}
[-H, H]
&=[-H, H_1-2d]\cup (H_1-2d, H_1']\cup (H_1', H_1'']\\
&\cup (H_1'', H_2-2d]\cup (H_2-2d, H_2+d]\cup (H_2+d, H].
\end{align*}
Note that some intervals could be empty. We estimate the summation of the part of $(H_2+d, H]$ in Lemma \ref{lemma:From_H2+d_to_H}, $(H_2-2d, H_2+d]$ in Lemma \ref{Lemma:H2}, $[-H, H_1-2d]$ in Lemma \ref{Lemma:H1-f1}, $(H_1-2d, H_1']$ in Lemma \ref{Lemma:H1-f2}, $(H_1', H_1'']$ in Lemma \ref{Lemma:From_H1'_to_H1''}, and $(H_1'', H_2-2d]$ in Section \ref{section:Applying_the_stationary_phase_method}, respectively.

In order to apply Lemma~\ref{lemma:exponential-int-first},  let 
\[
G(u)=u^{1/2} f'_{m,h}(u)= u^{1/2} \left(d \log\left(\frac{m}{d u} \right)+\theta-h\right). 
\]
Then we have
\[
G'(u) = \frac{1}{2u^{1/2}} f'_{m,h}(u) + u^{1/2} f''_{m,h}(u) = \frac{1}{2u^{1/2}}  \left(d \log\left(\frac{m}{d e^2 u} \right) +\theta-h \right). 
\]
Therefore, the following properties hold:
\begin{center}
(G1) $G(u)>0$ $\iff$ $u<\xi_{m,h}$,\quad\quad (G2) $G'(u)>0$ $\iff$ $u< e^{-2}\xi_{m,h}$.  
\end{center}

\begin{lemma}\label{lemma:From_H2+d_to_H}
We have
\[
\sum_{m=1}^M \frac{1}{m^{1-\sigma}}  \sum_{H_2+d< h\leq H} \int_{X}^N \frac{ e(f_{m,h}(u))}{u^{1/2}}  du \ll \begin{cases}  
X^{-1/2}(\log N) \log \log N&  \text{if $\sigma=0$}; \\
X^{-1/2}\log\log N &  \text{if $\sigma<0$}.
\end{cases}
\]
\end{lemma}
\begin{proof} Take any integer $h$ with $H_2+d< h\leq H$. Then $X> \xi_{m,h} e$ holds by the definition of $H_2$. Thus, we have 
\begin{equation}\label{Inequality:large-h} 
\xi_{m,h} e^{-2}<\xi_{m,h}< \xi_{m,h} e  < X.   
\end{equation}
By \eqref{Inequality:large-h}, (G1), and (G2), the function $G(u)$ is negative and monotonically decreasing on  $[X, N]$, and hence for every $u\in[X,N]$,
\[
|G(u)|= -G(u) \geq -G(X)= X^{1/2} \left( h-\theta- d \log\left( \frac{m}{d X} \right) \right)= X^{1/2} (h-H_2) .
\]
Therefore, by Lemma~\ref{lemma:exponential-int-first} with $g(u)\coloneqq u^{-1/2}$, $f(u)\coloneqq f_{m,h}(u)$, $a\coloneqq X$, and $b\coloneqq N$, we have
\begin{align*}
\sum_{m=1}^M \frac{1}{m^{1-\sigma}}  \sum_{H_2+d\leq h\leq H} \int_{X}^N \frac{ e(f_{m,h}(u))}{u^{1/2}}  du &\ll \sum_{m=1}^M \frac{1}{m^{1-\sigma}} \sum_{H_2+d< h\leq H}  \frac{1}{X^{1/2} (h-H_2)}\\
&\ll \begin{cases}  
X^{-1/2}(\log N) \log \log N&  \text{if $\sigma=0$}; \\
X^{-1/2}\log\log N &  \text{if $\sigma<0$}.
\end{cases}
\end{align*}
\end{proof}

\begin{lemma}\label{Lemma:zeta-critical2} For all $M\in \mathbb{N}$ and $|t|\ge 1$, we have
\[
\zeta (1-\sigma+it) =\sum_{m=1}^M \frac{1}{m^{1-\sigma+it}} -it \int_{M}^\infty \frac{\{u\}}{u^{2-\sigma+it}} du +O(M^\sigma |t|^{-1} +M^{-1+\sigma} )  
\]
\end{lemma}

\begin{proof}
If $\Re s >1$, then the Abel summation formula leads to
\[
\zeta(s)= \sum_{m=1}^M \frac{1}{m^s} + \frac{M^{1-s}}{s-1} -s \int_{M}^\infty 
 \frac{\{u\}}{u^{s+1}} du.  
\]
The integral of the third term is convergent for $\Re s>0$. Thus, by substituting $s=1-\sigma+it$, 
\begin{align*}
\zeta(1-\sigma+it) &=\sum_{m=1}^M  \frac{1}{m^{1-\sigma+it}} + \frac{M^{\sigma-it}}{it-\sigma} - (1-\sigma+it) \int_{M}^\infty 
 \frac{\{u\}}{u^{2-\sigma+it}} du\\
 &= \sum_{m=1}^M \frac{1}{m^{1-\sigma+it}} -it \int_{M}^\infty \frac{\{u\}}{u^{2-\sigma+it}} du + O (M^\sigma |t|^{-1} +M^{-1+\sigma} ).  
\end{align*}
\end{proof}

\begin{lemma}\label{Lemma:e-zeta-interval} For all real numbers $a,b\geq 1$ and $u\geq 1$, we have
\[
\left|\sum_{a< m\leq b} \frac{e(f_{m,h}(u))}{m^{1-\sigma}} \right| \ll \min(u a^{-1+\sigma} +  a^{\sigma}u^{-1},  \log(b/a)+1).
\]
\end{lemma}

\begin{proof}
It follows that
\[
\left|\sum_{a< m\leq b} \frac{e(f_{m,h}(u))}{m^{1-\sigma}} \right|=  \left|\sum_{a< m\leq b} \frac{e(d u \log (em/(d u) ) +\theta u-hu )  }{m^{1-\sigma}} \right| = \left|\sum_{a< m\leq b} \frac{1}{m^{1-\sigma+2\pi idu} } \right|. 
\]
By applying Lemma~\ref{Lemma:zeta-critical2}, we have 
\begin{align*}
\sum_{a< m\leq b} \frac{1}{m^{1-\sigma+2\pi idu} } &= \sum_{ 1\leq m\leq b} \frac{1}{m^{1-\sigma+2\pi idu} } - \sum_{1\leq m\leq a} \frac{1}{m^{1-\sigma+2\pi idu} } \\
&= \zeta(1-\sigma+2\pi idu) -\zeta (1-\sigma+2\pi idu)\\
&\quad-2\pi idu \left(\int_{a}^\infty \frac{\{u\}}{u^{2-\sigma+2\pi idu}} du -  \int_{b}^\infty \frac{\{u\}}{u^{2-\sigma+2\pi idu}} du\right) \\ 
&\quad+O(b^\sigma u^{-1} +b^{-1+\sigma} + a^\sigma u^{-1} +a^{-1+\sigma})  \\
&\ll u a^{-1+\sigma} + a^\sigma u^{-1}.
\end{align*}
On the other hand, we get
\begin{align*}
    \left|\sum_{a< m\leq b} \frac{e(f_{m,h}(u))}{m^{1-\sigma}} \right|\leq \sum_{a<m\leq b} \frac{1}{m^{1-\sigma}}
    \ll
    \begin{cases}
        \log(b/a)+1  &  \text{if $\sigma=0$},\\
        1  &  \text{if $\sigma<0$}.
    \end{cases}
\end{align*}
\end{proof}

\begin{lemma}\label{Lemma:H2} We have
\[
\sum_{m=1}^M \frac{1}{m^{1-\sigma}}  \sum_{H_2-2d\leq h\leq H_2+d} \int_{X}^N \frac{ e(f_{m,h}(u))}{u^{1/2}}  du \ll \begin{cases}
 X^{1/2}\log X+  X^{-1/2}\log N  & \text{if $\sigma=0$}; \\ 
 1 & \text{if $\sigma<0$}.
\end{cases}
\]
\end{lemma}

\begin{proof}
Take any integer $h$ with $H_2-2d\leq h \leq H_2+d$. Then we see that 
\begin{align} \label{eq:restrictH2-1}
H_2-2d \leq h \leq H_2+d &\iff e^{-1 }X \leq \xi_{m,h} \leq e^2 X   \\ \notag
&\iff d e^{(h-\theta)/d} Xe^{-1} \leq m \leq d e^{(h-\theta)/d} Xe^2.  
\end{align} 
We decompose the integral as 
\begin{align*}
\int_{X}^N \frac{ e(f_{m,h}(u))}{u^{1/2}}  du = \int_{X}^{e^{3 }X} \frac{ e(f_{m,h}(u))}{u^{1/2}}  du+\int_{e^3 X}^{N} \frac{ e(f_{m,h}(u))}{u^{1/2}}  du \eqqcolon \mathcal{T}_1+\mathcal{T}_2.
\end{align*}
Let us evaluate $\mathcal{T}_2$. For every $u\in [e^{3}X, N]$, by \eqref{eq:restrictH2-1}, it follows that 
\[
e^{-2} \xi_{m,h} < \xi_{m,h} \leq e^2X < e^3X \leq u,
\]
and hence by (G1) and (G2), the function $G(u)$ is negative and monotonically decreasing on $[e^{3}X, N]$. Therefore, we have
\begin{align*}
|G(u)|&= -G(u) \geq -G(e^3X)=  e^{3/2}X^{1/2} (h+3d-H_2)\gg X^{1/2}.
\end{align*}
Therefore, Lemma~\ref{lemma:exponential-int-first} with $g(u)\coloneqq u^{-1/2}$, $f(u)\coloneqq f_{m,h}(u)$, $a\coloneqq e^{3}X$, and $b\coloneqq N$ implies that 
\begin{align*}
\sum_{m=1}^M \frac{1}{m^{1-\sigma}}\sum_{H_2-2d\leq h\leq H_2+d}\mathcal{T}_2 
&\ll \sum_{m=1}^M \frac{1}{m^{1-\sigma}} \sum_{H_2-2d\leq h\leq H_2+d}  \frac{1}{X^{1/2} }\\
&\ll \begin{cases}
X^{-1/2}\log M &\text{if $\sigma=0$;}\\
1 &  \text{if $\sigma<0$.}
\end{cases}   
\end{align*}
Let us evaluate $\mathcal{T}_1$. In the case $\sigma<0$, by \eqref{eq:restrictH2-1}, each $u \in [X,e^{3}X]$ satisfies 
\[
e^{-2} \xi_{m,h}\leq X\leq u.
\]
Thus, by (G2), $G(u)$ is monotonically decreasing on $[X,e^{3}X]$. Further, $|f''_{m,h}(u)|\gg u^{-1} \gg  X^{-1}$ holds on $[X, e^{3}X]$, where the implicit constant does not depend on $m$ or $h$.  Therefore, Lemma~\ref{lemma:second-derivative} with $g(u) \coloneqq u^{-1/2}$, $f(u)\coloneqq f_{m,h}(u)$, $a\coloneqq X$, and $b\coloneqq e^{3}X $  implies that 
\[
\int_{X}^{e^3 X} \frac{ e(f_{m,h}(u))}{u^{1/2}}  du  \ll X^{1/2} X^{-1/2}\ll 1.
\]
Therefore, if $\sigma<0$, then we have 
\[
\sum_{m=1}^M \frac{1}{m^{1-\sigma}}  \sum_{H_2-2d\leq h\leq H_2+d} \int_{X}^N \frac{ e(f_{m,h}(u))}{u^{1/2}}  du \ll 1.
\]
In the case $\sigma=0$, we switch the double summation as follows:   
\begin{align*}
\sum_{m=1}^M \frac{1}{m^{1-\sigma}}\sum_{H_2-2d\leq h\leq H_2+d}\mathcal{T}_1= \sum_{H_2(1)-2d \leq h \leq H_2(M)+d  } \sum_{a\leq m\leq b} \frac{1}{m} \int_{X}^{e^{3} X} \frac{ e(f_{m,h}(u))}{u^{1/2}}  du,
\end{align*}
where  $a= d e^{(h-\theta)/d} Xe^{-1} $ and  $ b=\min ( d e^{(h-\theta)/d} Xe^2,M)$. Remark that $\log (b/a)\leq 3$. We see that 
\begin{gather*}
a\geq 1 \iff d e^{(h-\theta)/d} Xe^{-1}\geq 1 \iff h \geq d \log ( 1/(d X)) +\theta+d = H_2(1)+d
\end{gather*}
and
\begin{gather*}
d e^{(h-\theta)/d} Xe^2\leq M \iff h \leq d \log ( M/(d X))+\theta -2d = H_2(M)-2d.
\end{gather*}
By Lemma~\ref{Lemma:e-zeta-interval}, we have
\begin{align*}
    &\sum_{H_2(1)+d\leq h\leq H_2(M)+d} \sum_{a\leq m\leq b} \frac{1}{m} \int_X^{e^3 X} \frac{e(f_{m,h}(u))}{u^{1/2}} du\\
    &= \left( \sum_{H_2(1)+d\leq h\leq 0} +\sum_{0< h\leq H_2(M)+d} \right) \int_X^{e^3X} u^{-1/2} \sum_{a\leq m\leq b} \frac{e(f_{m,h}(u))}{m} du\\
    &\ll\sum_{H_2(1)+d\leq h\leq 0} \int_X^{e^3 X} u^{-1/2} (\log (b/a) +1) du +\sum_{0< h\leq H_2(M)+d} \int_X^{e^3X} (u^{1/2}a^{-1}+u^{-3/2}) du\\
    &\ll \sum_{H_2(1)+d\leq h\leq 0} X^{1/2} +\sum_{0< h\leq H_2(M)+d} (X^{1/2} e^{-h/d} +X^{-1/2})\\
    &\ll X^{1/2} |H_2(1)| +X^{1/2} +X^{-1/2}\log M \ll X^{1/2}\log X +X^{-1/2}\log N.
\end{align*}
\begin{comment}
If $H_2(1)+d\leq h \leq H_2(M)+2d$, then Lemma~\ref{Lemma:e-zeta-interval} implies that
\begin{align*}
\sum_{a< m\leq b} \frac{1}{m} \int_{X}^{e^{3} X} \frac{ e(f_{m,h}(u))}{u^{1/2}}  du &\ll \int_{X}^{e^{3}X} (u^{1/2} a^{-1} + u^{-3/2}
) du \ll X^{3/2}a^{-1} + X^{-1/2},
\end{align*}
which leads to 
\begin{align*}
&\sum_{H_2(1)+d \leq h\leq H_2(M)+2d}  \sum_{a< m\leq b} \frac{1}{m} \int_{X}^{e^{3} X} \frac{ e(f_{m,h}(u))}{u^{1/2}}  du \\
&\ll  X^{1/2} e^{-H_2(1)/d} + X^{-1/2}\log M \ll X^{3/2} + X^{-1/2} \log M.
\end{align*}
\end{comment}
If $H_2(1)-2d\leq h< H_2(1)+d$, then $a,b\ll 1$, which implies that 
\[
\sum_{H_2(1)-2d < h\leq H_2(1)+d}  \sum_{a\leq m\leq b} \frac{1}{m} \int_{X}^{e^{3} X} \frac{ e(f_{m,h}(u))}{u^{1/2}}  du\ll \sum_{H_2(1)-2d < h\leq H_2(1)+d}  \sum_{a\leq m\leq b} \frac{X^{1/2}}{m}\ll X^{1/2} .
\]
Combining the above estimates, we conclude Lemma~\ref{Lemma:H2}.
\end{proof}

\begin{lemma}\label{Lemma:H1-f1} We have
\[
\sum_{m=1}^M \frac{1}{m^{1-\sigma}}  \sum_{-H\leq h\leq H_1-2d} \int_{X}^N \frac{ e(f_{m,h}(u))}{u^{1/2}}  du \ll 
\begin{cases}
    X^{-1/2}(\log N)(\log\log N)  & \text{ if } \sigma=0;\\
    X^{-1/2}(\log\log N)  &  \text{ if } \sigma<0.
\end{cases}
\]
\end{lemma}
\begin{proof} Take any integer $h$ with $-H\leq h\leq H_1-2d$. In this case, $\xi_{m,h}\geq e^{2} N$ holds. Thus, we have 
\begin{equation}\label{Inequality:small-h} 
  \xi_{m,h}> \xi_{m,h} e^{-2}   \geq  N.   
\end{equation}
By (G1) and (G2), the function $G(u)$ is positive and monotonically increasing on  $[X, N]$, and hence
\[
|G(u)|= G(u) \geq G(X)= X^{1/2} \left( d \log\left( \frac{m}{d X}\right)+\theta-h \right)= X^{1/2} (H_2-h).
\]
Therefore, by Lemma~\ref{lemma:exponential-int-first} with $g(u)\coloneqq u^{-1/2}$, $f(u)\coloneqq f_{m,h}(u)$, $a\coloneqq X$, and $b\coloneqq N$, 
\begin{align*}
&\sum_{m=1}^M \frac{1}{m^{1-\sigma}}   \sum_{-H\leq h\leq H_1-2d} \int_{X}^N \frac{ e(f_{m,h}(u))}{u^{1/2}}  du \\
&\ll \sum_{m=1}^M \frac{1}{m^{1-\sigma}} \sum_{-H\leq h\leq H_1-2d}  \frac{1}{X^{1/2} (H_2-h)}\\
&\ll 
\begin{cases}
    X^{-1/2}(\log N)(\log\log N)  & \text{ if } \sigma=0;\\
    X^{-1/2}(\log\log N)  &  \text{ if } \sigma<0.
\end{cases} 
\end{align*}
\end{proof}

\begin{lemma}\label{Lemma:H1-f2} Suppose that $N^{-1/2} \leq V$. Then we have
\[
    \sum_{m=1}^M \frac{1}{m^{1-\sigma}}  \sum_{H_1-2d< h\leq H_1'} \int_{X}^{N} \frac{ e(f_{m,h}(u))}{u^{1/2}}  du \ll  1+\ichi(\sigma=0) \frac{\log N}{N^{1/2} V }.
\]
\end{lemma}

\begin{proof}
Take any integer $h$ with $H_1-2d< h\leq H_1'$. We note that $H_1'-H_1=d\log(1+V)\asymp V$. Further, we see that 
\begin{align}\label{equivalent:h-xi}
  H_1-2d< h< H_1' &\iff N(1+V)< \xi_{m,h}< Ne^2, 
\end{align} 
and hence $e^{-2}\xi_{m,h}<N$. We decompose the integration as follows:
\[
\int_{X}^{N} \frac{ e(f_{m,h}(u))}{u^{1/2}}  du= \int_{X}^{e^{-2}\xi_{m,h}} + \int_{e^{-2}\xi_{m,h}}^N  \eqqcolon \mathcal{T}_1 +\mathcal{T}_2.
\]
For $\mathcal{T}_1$, by (G1) and (G2), the function $G(u)$ is positive and monotonically increasing on $[X, e^{-2}\xi_{m,h} ]$, which leads to 
\begin{align*}
|G(u)|&= G(u) \geq G(X)=  X^{1/2} \left(d\log \left(\frac{m}{dX}\right) + \theta-h\right)=X^{1/2}(H_2-h)\gg X^{1/2}\log N.
\end{align*}
Therefore, by Lemma~\ref{lemma:exponential-int-first} with $g(u)\coloneqq u^{-1/2}$, $f(u)\coloneqq f_{m,h}(u)$, $a\coloneqq X$, and $b\coloneqq e^{-2}\xi_{m,h}$, 
\begin{align*}
\sum_{m=1}^M \frac{1}{m^{1-\sigma}}  \sum_{H_1-2d< h\leq H_1'} \mathcal{T}_1 \ll \sum_{m=1}^M \frac{1}{m^{1-\sigma}} \sum_{H_1-2d\leq h\leq H_1'} \frac{1}{X^{1/2} \log N } \ll 1.    
\end{align*}
For  $\mathcal{T}_2$, by (G1), (G2), and \eqref{equivalent:h-xi}, the function $G(u)$ is positive and monotonically decreasing on $[e^{-2}\xi_{m,h}, N ]$, and hence for every integer $m\in[1,M]$ and $h\in [H_1-2d, H_1']$, we obtain  
\begin{align*}
|G(u)|&= G(u) \geq G(N)=  N^{1/2} \left(d\log \left(\frac{m}{dN}\right) + \theta-h\right)\\
&= N^{1/2} \left(d\log \left(\frac{m}{d(1+V)N}\right) + d\log (1+V) + \theta-h\right)\\
&=N^{1/2}(H_1'+d\log(1+V)-h) \gg N^{1/2} V.
\end{align*}
Therefore, by Lemma~\ref{lemma:exponential-int-first} with $g(u)\coloneqq u^{-1/2}$, $f(u)\coloneqq f_{m,h}(u)$, $a\coloneqq e^{-2}\xi_{m,h}$, and $b\coloneqq N$, 
\begin{align*}
\sum_{m=1}^M \frac{1}{m^{1-\sigma}}  \sum_{H_1-2d< h\leq H_1'} \mathcal{T}_2
&\ll \sum_{m=1}^M \frac{1}{m^{1-\sigma}} \sum_{H_1-2d\leq h\leq H_1'} \frac{1}{N^{1/2} V } \\
&\ll \begin{cases}
N^{-1/2}V^{-1} \log N & \text{if $\sigma=0$;} \\
1 & \text{if $\sigma<0$,}
\end{cases}
\end{align*}
where we should choose $V$ as $N^{-1/2}\leq V$. 

\end{proof}
%%%%%%%%%%%%

\begin{lemma}\label{Lemma:Vaaler}
Let $J>0$ be an integer. There exist $a_j\in \mathbb{C}$ $(1\leq |j|\leq J)$ with $a_j \ll 1$ such that we have
\[
\psi^{*} (x) = - \sum_{1\leq |j|\leq J} a_j \frac{e(jx) }{2\pi i j}  
\]
and 
\[
|\psi(x) -\psi^{*}(x)  |\leq \frac{1}{2J+2} \sum_{|j|\leq J} \left(1-\frac{|j|}{J+1}\right) e(jx).  
\]
\end{lemma}

\begin{proof}
See \cite[Theorem A.6]{Graham-Kolesnik}.
\end{proof}

\begin{lemma}\label{Lemma:From_H1'_to_H1''} We have
\[
    \sum_{m=1}^M \frac{1}{m^{1-\sigma}}  \sum_{H_1'< h\leq H_1''} \int_{X}^{N} \frac{ e(f_{m,h}(u))}{u^{1/2}}  du \ll  \begin{cases}
    V \log N +(\log \log N)^2 & \text{if $\sigma=0$;} \\
    1 & \text{if $\sigma<0$.}   
    \end{cases}
\]
\end{lemma}

\begin{proof}
Take any integer $h$ with $H_1'< h\leq H_1''$. Then it follows that 
\begin{align*}
  H_1'< h\leq H_1'' &\iff N(1-V)\leq \xi_{m,h}< N(1+V). 
  \end{align*} 
We decompose the integral as 
\begin{align*}
\int_{X}^N \frac{ e(f_{m,h}(u))}{u^{1/2}}  du = \int_{X}^{e^{-2}\xi_{m,h}} \frac{ e(f_{m,h}(u))}{u^{1/2}}  du+\int_{e^{-2} \xi_{m,h}}^{N} \frac{ e(f_{m,h}(u))}{u^{1/2}}  du \eqqcolon \mathcal{T}_1+\mathcal{T}_2.
\end{align*}

For every $m\in \mathbb{Z}_{[1,M]}$ and $h\in \mathbb{Z}_{(H_1',H_1'']}$, the function $G(u)$ is positive and monotonically increasing on $[X, e^{-2}\xi_{m,h}]$, and hence 
\begin{align*}
|G(u)|&= G(X) \geq X^{1/2}(H_2-h)  \gg X^{1/2}\log N .
\end{align*}
Therefore, Lemma~\ref{lemma:exponential-int-first} with $g(u)\coloneqq u^{-1/2}$, $f(u)\coloneqq f_{m,h}(u)$, $a\coloneqq X$, and $b\coloneqq e^{-2}\xi_{m,h} $ yields that 
\begin{align*}
&\sum_{m=1}^M \frac{1}{m^{1-\sigma}}  \sum_{H_1'< h\leq H_1''} \int_{X}^{N} \frac{ e(f_{m,h}(u))}{u^{1/2}}  du \ll \sum_{m=1}^M \frac{1}{m^{1-\sigma}} \sum_{H_1' < h\leq H_1''} \frac{1}{X^{1/2}\log N } \ll 1.    
\end{align*}
For evaluating upper bounds for $\mathcal{T}_2$, we have $f''_{m,h}(u)=-d/u \asymp N^{-1}$ on $u\in [e^{-2}\xi_{m,h} ,N]$ since $\xi_{m,h}\asymp N$. Therefore, by Lemma~\ref{lemma:second-derivative}, we obtain 
\begin{align*}
\sum_{m=1}^M \frac{1}{m^{1-\sigma}}\sum_{H_1'<h\leq H_1''}\mathcal{T}_2&\ll \sum_{m=1}^M \frac{1}{m^{1-\sigma}}\sum_{H_1'< h\leq H_1''} N^{-1/2} N^{1/2}\leq \sum_{m=1}^M \frac{\lfloor H_1''\rfloor -\lfloor H_1'\rfloor }{m^{1-\sigma}}\\
&= \sum_{m=1}^M \frac{H_1''-H_1'}{m^{1-\sigma}} -  \sum_{m=1}^M \frac{\psi(H_1'')  -\psi(H_1')}{m^{1-\sigma}} \eqqcolon \mathcal{T}_{21}+\mathcal{T}_{22}, 
\end{align*}
say. In the case $\sigma<0$, we have $\mathcal{T}_{21}\ll 1$ and $\mathcal{T}_{22}\ll 1$. In the case $\sigma=0$, by the definitions of $H_1'$ and $H_1''$, we obtain 
\[
H_1''-H_1'= d\log\left(\frac{1+V}{1-V} \right) \asymp V.
\]
Thus, we have $\mathcal{T}_{21}\ll V \log M$. 

Let $\tilde{H}= H_1''$ or $H_1'$. Let $J$ be a positive integer with $J\leq N$ determined later. By Lemma~\ref{Lemma:Vaaler}, we see that 
\[
\sum_{m=1}^M \frac{\psi(\tilde{H})}{m} = \sum_{m=1}^M \frac{\psi^*(\tilde{H}) }{m}  + O\left(\left|\sum_{m=1}^M\frac{1}{m}\cdot  \frac{1}{J} \sum_{|j|\leq J} \left(1- \frac{|j|}{J+1} \right) e(j \tilde{H} )  \right|\right).
\]
Applying Lemma~\ref{Lemma:e-zeta}, the error term is upper-bounded by
\begin{align*}
\frac{1}{J} \sum_{|j|\leq J} \left|\sum_{m=1}^M \frac{e(j \tilde{H} )}{m}  \right|
&=\frac{1}{J} \sum_{|j|\leq J} \left|\sum_{m=1}^M \frac{e(j d\log m )}{m}  \right|\\
&\ll \frac{1}{J}\left( \log M +\sum_{1\leq |j| \leq J} \log^{2/3}(J+1) \right) \\
&\ll  \frac{\log M}{J} +\log(J+1).
\end{align*}
Similarly, the absolute value of the first sum is upper-bounded by
\begin{align*}
\sum_{1\leq |j|\leq J} \frac{1}{j} \left|\sum_{m=1}^M \frac{e(j d \log m )}{m} \right| \ll \sum_{1\leq |j|\leq J} \frac{\log (j+1)}{j}\ll \log^2 (J+1).
\end{align*}
Therefore, by choosing $J = \lfloor \log M \rfloor$, we obtain $\mathcal{T}_{22}\ll (\log\log M)^2$. This completes the proof of the lemma.
\end{proof}

\section{Applying the stationary phase method}\label{section:Applying_the_stationary_phase_method}
Let $N, M\in \mathbb{Z}_{\ge 1}$ with $C_1N\leq M \leq C_2 N$, where $C_1$ and $C_2$ are constants satisfy $C_2>C_1>1$. Let $V\geq N^{-1/2}$. Recall that $X=X(\sigma)$ is a parameter of integers with $3\leq X\leq \min (M,N)$. By combining \eqref{eq:zeta-restrictH}, Lemmas~\ref{lemma:From_H2+d_to_H}, \ref{Lemma:H2}, \ref{Lemma:H1-f1},  \ref{Lemma:H1-f2}, and \ref{Lemma:From_H1'_to_H1''}, we obtain 
\begin{align*}
&\sum_{1\leq n\leq N} \zeta(\sigma+idn) \frac{e(\theta n)}{n^{1-\sigma}} \\
&=   D \sum_{m=1}^M \frac{1}{m^{1-\sigma}}   \sum_{H_1''\leq h<H_2-2d } \int_{X}^N \frac{e(f_{m,h}(u))}{u^{1/2}} du  \\
&\quad+ O(X^{1/2} \log \log N) \\
&\quad+ O\left( \ichi(\sigma=0) \left(  \log^{1-\alpha}N+ X^{1/2} (\log X)^{2/3} \log \log N+ \frac{\log \log N\log N  }{X^{1/2} }\right) \right) \\
&\quad+O\left(\ichi(\sigma=0) \left( X^{1/2}\log X + (\log\log N)^2 +V\log N + \frac{\log N}{N^{1/2}V }\right) \right ).
\end{align*}

The aim of this section is to prove the following lemma.

\begin{lemma}\label{lemma:main-term} We have 
\begin{align*}
&\sum_{m=1}^M  \frac{1}{m^{1-\sigma}} \sum_{H_1''\leq h<H_2-2d} \int_{X}^N \frac{e(f_{m,h}(u) )}{u^{1/2}} du \\
&=   \frac{e(-1/8)}{d^{1/2}}  \sum_{m=1}^M \frac{1}{m^{1-\sigma}} \sum_{H_1'' \leq h<H_2-2d} e(m e^{(\theta-h)/d}) \\
&\quad+ O\left( (1 + \ichi(\sigma=0) \log N) \left( \frac{\log\log N}{X^{1/2}}+ \frac{1}{X^{1/2}V}  +  \frac{1}{X^{3/2}V^{3}}\right) \right).
\end{align*}
\end{lemma}

\begin{lemma}[Stationary phase integral]\label{Lemma:Huxley}
Let $f(u)$ be a real function, four times continuously differentiable and let $g(u)$ be a real function, three times continuously differentiable on the interval $[a,b]$. Let $T_1$, $T_2$, $L_1$, $L_2$, $B$ be positive parameters with $L_1\geq b-a$. Suppose that for all $x\in [a,b]$ we have
\[
f''(u) \leq  -T_1/ (B^2L_1^2)
\]
and
\[
|f^{(r)} (u)| \leq B^r T_1/L_1^r ,\quad |g^{(s)}(u)| \leq B^s T_2/L_2^s
\]
for $r=2,3,4$ and $s=0,1,2$. Assume that $f'(u)$ changes sign from positive to negative at $u=c$ with $a<c<b$. If $T_1$ is sufficiently large in terms of $B$, then we have  
\begin{align*}
\int_{a}^b g(u) e (f(u)) du &= \frac{g(c)e(f(c) -1/8)}{\sqrt{|f''(c)| } } + \frac{g(b) e(f(b))}{ 2\pi i f'(b)} -\frac{g(a) e(f(a))}{ 2\pi i f'(a)} \\
&+O\left( \frac{B^4 L_1^4 T_2 }{T_1^2 } ((c-a)^{-3} +(b-c)^{-3} )  \right) \\
&+O\left( \frac{B^{13} L_1 T_2 }{T_1^{3/2} } \left(1 + \frac{L_1}{ B^4 L_2} \right)^2  \right)  .
\end{align*}
\end{lemma}
\begin{proof}
    See \cite[Theorem 2]{Huxley}.
\end{proof}

\begin{proof}[Proof of Lemma~\ref{lemma:main-term}] 
Let $U\in (0,1/2]$ be a positive parameter satisfying $(1+U)(1-V)\leq 1$. We observe that 
\begin{align*}
H_1''\leq  h < H_2-2d &\iff Xe^{2} < \xi_{m,h} \leq N(1-V).    
\end{align*}
Let $m\in \mathbb{Z}_{[1,M]}$ and $h\in \mathbb{Z}_{[H_1'', H_2-2d)}$, and set $\xi=\xi_{m,h}$. We decompose the integral as 
\begin{align*}
&\int_{X}^N \frac{e(f_{m,h}(u) )}{u^{1/2}} du \\
&= \left(\int_{X}^{e^{-2}\xi} + \int_{e^{-2}\xi }^{ (1-U)\xi } + \int_{(1-U)\xi} ^{(1+U)\xi} +\int_{(1+U)\xi}^N \right) \frac{e(f_{m,h}(u) )}{u^{1/2}} du \\
&\eqqcolon \mathcal{I}_1+\mathcal{I}_2+\mathcal{I}_3+\mathcal{I}_4,
\end{align*}
say. By combining $(1+U)(1-V)\leq 1$ and $\xi \leq N(1-V)$, we see that $(1+U)\xi\leq N$. 

For $\mathcal{I}_1$, by (G1) and (G2), the function $G(u)$ is positive and monotonically increasing on $[X, e^{-2}\xi]$, we have 
\[
|G(u)| \geq G(X) = X^{1/2} \left(H_2-h \right)
\]
for every $u\in[X,e^{-2}\xi]$. Therefore, Lemma~\ref{lemma:exponential-int-first} with $g(u)\coloneqq u^{-1/2}$, $f(u)\coloneqq f_{m,h}(u)$, $a\coloneqq X$, and $b\coloneqq e^{-2}\xi $  implies that
\[
\mathcal{I}_1 \ll X^{-1/2}(H_2-h)^{-1} . 
\]
For $\mathcal{I}_2$, by (G1) and (G2) the function $G(u)$ is positive and monotonically decreasing on $[e^{-2}\xi, (1-U)\xi]$, and hence if $u$ belongs to the interval, then 
\begin{align*}
|G(u)| =G(u)\geq G((1-U)\xi )\gg \xi^{1/2}\left( d \log \left(\frac{m}{d(1-U)\xi  } \right) + \theta -h \right)\gg U\xi^{1/2}
\end{align*}
since $\xi=(m/d)e^{(\theta-h)/d}$. Therefore, Lemma~\ref{lemma:exponential-int-first}  with $g(u)\coloneqq u^{-1/2}$, $f(u)\coloneqq f_{m,h}(u)$, $a\coloneqq e^{-2}\xi$, and $b\coloneqq (1-U)\xi$ implies that
\[
\mathcal{I}_2 \ll \xi^{-1/2} U^{-1}. 
\]
For $\mathcal{I}_4$, the function $G(u)$ is negative and monotonically decreasing on $[(1+U)\xi, N]$. Thus, each $u\in [(1+U)\xi, N]$ satisfies 
\[
|G(u)|= -G(u)\geq -G((1+U)\xi)= \xi^{1/2}  \left(h-\theta - d \log \left(\frac{m}{d (1+U) \xi } \right) \right)\gg \xi^{1/2}U . 
\]
Therefore, Lemma~\ref{lemma:exponential-int-first} with $g(u)\coloneqq u^{-1/2}$, $f(u)\coloneqq f_{m,h}(u)$, $a\coloneqq (1+U)\xi$, and $b\coloneqq N$ implies that
\[
\mathcal{I}_4 \ll  \xi^{-1/2} U^{-1}. 
\]
For $\mathcal{I}_3$, we shall apply the stationary phase integral. Let $g(u)=u^{-1/2}$. It follows that $f^{(r)}_{m,h}=(-1)^{r-1} (r-2)! d u^{-r+1}$ for every $r=2,3,4$. Therefore, for all $u\in [(1-U)\xi, (1+U) \xi ]$, we obtain 
\[
|f^{(2)}_{m,h}(u) |= \frac{d}{u} \asymp \frac{1}{\xi}, \quad |f^{(r)}_{m,h}(u) |\asymp \frac{1}{\xi^{r-1} }. 
\]
Thus, we may choose $T_1= \xi_{m,h}$ and $L_1=\xi$. Note that $L_1 \geq (1+U)\xi - (1-U)\xi$ holds since $U\in (0,1/2]$. Further, each $u\in [(1-U)\xi, (1+U) \xi ]$ satisfies
\[
|g(u)| \asymp \xi^{-1/2},\quad |g'(u)| \asymp \xi^{-3/2},  \quad |g''(u)|\asymp \xi^{-5/2}, 
\]
and hence, we may take $T_2=\xi^{-1/2}$ , $L_2=\xi$, and $B>0$ as a sufficiently large absolute constant. Let $c= \xi$. Since $X\ll \xi=T_1$ and $X$ can be taken sufficiently large in terms of absolute constants, applying Lemma~\ref{Lemma:Huxley} with $a=(1-U)\xi$ and $b=(1+U)\xi$, we get
\begin{align*}
\mathcal{I}_3&= \frac{g(c)e(f(c) -1/8)}{\sqrt{|f''(c)| } } + \frac{g(b) e(f(b))}{ 2\pi i f'(b)} -\frac{g(a) e(f(a))}{ 2\pi i f'(a)} \\
&+O\left( \frac{B^4 L_1^4 T_2 }{T_1^2 } ((c-a)^{-3} +(b-c)^{-3} )  \right) +O\left( \frac{B^{13} L_1 T_2 }{T_1^{3/2} } \left(1 + \frac{L_1}{ B^4 L_2} \right)^2  \right).  
\end{align*}
The first term is equal to 
\[
\frac{1}{\xi^{1/2}} e(m e^{(\theta-h)/d} -1/8) \frac{\xi^{1/2}}{d^{1/2}} =\frac{e(-1/8)}{d^{1/2} }  e(me^{(\theta-h)/d}). 
\]
The second and third terms are upper bounded by 
\[
\ll \frac{1}{\xi^{1/2}} \left( \frac{1}{|f'((1+U)\xi ) |} +\frac{1}{|f'((1-U)\xi ) |}  \right) \ll \xi^{-1/2} U^{-1}.  
\]
The error terms are also bounded by
\[
\ll \frac{\xi^4 \xi^{-1/2} }{\xi^2} \frac{1}{\xi^3 U^3} + \frac{\xi \xi^{-1/2}}{\xi^{3/2} } \left(1+\frac{\xi}{\xi}\right)^2 \ll \frac{1}{\xi^{3/2}U^3} + \frac{1}{\xi}.   
\]
Therefore, we obtain 
\[
\mathcal{I}_1+\mathcal{I}_2+\mathcal{I}_3+\mathcal{I}_4=\frac{e(-1/8)}{d^{1/2}} e(m e^{(\theta-h)/d}) +E(m,h), 
\]
where
\[
E(m,h)\ll X^{-1/2} (H_2-h)^{-1} +\frac{1}{\xi^{1/2}U}+\frac{1}{\xi^{3/2} U^{3}}+\frac{1}{\xi}.
\]
Therefore, by choosing $U=V$, we have $(1+U)(1-V)\leq 1$ and  
\begin{align*}
&\sum_{m=1}^M \frac{1}{m^{1-\sigma}} \sum_{H_1'' \leq h< H_2-2d}  E(m,h)\\
&\ll \sum_{m=1}^M  \frac{1}{m^{1-\sigma}} \sum_{ \substack{H_1''\leq h<H_2-2d  }}
 \left(\frac{1}{X^{1/2}(H_2-h)} + \frac{e^{(h-\theta)/(2d)}}{m^{1/2}V} + \frac{e^{3(h-\theta)/(2d)}}{m^{3/2}V^3}+\frac{e^{(h-\theta)/d}}{m} \right)\\
&\ll \left(1 + \ichi(\sigma=0) \log N  \right)\left(\frac{\log\log N}{X^{1/2}}+\frac{1}{X^{1/2}V}  +  \frac{1}{X^{3/2}V^{3}} + \frac{1}{X}\right).
\end{align*}
\end{proof}

At last, we optimize the error terms. By combining the above discussion, we obtain 
\begin{align*}
&\sum_{1\leq n\leq N} \zeta(2\pi idn) \frac{e(\theta n)}{n^{1-\sigma}} 
\\
&= d^{-\sigma}\sum_{m=1}^M \frac{1}{m^{1-\sigma}}   \sum_{H_1''<h<H_2-2d }  e(me^{(\theta-h)/d})  \\
&\quad+ O\left(X^{1/2} \log \log N +\frac{1}{X^{1/2}V}  +  \frac{1}{X^{3/2}V^{3}} + \frac{1}{X}  \right) \\
&\quad+ O\left( \ichi(\sigma=0) \left(  \log^{1-\alpha}N+ X^{1/2} (\log X)^{2/3} \log \log N+ \frac{\log N\log \log N  }{X^{1/2} }\right) \right) \\
&\quad+O\left(\ichi(\sigma=0) \left(  X^{1/2}\log X   + (\log\log N)^2 +V\log N + \frac{\log N}{N^{1/2}V }\right) \right )\\
&\quad+O \left( \ichi(\sigma=0)  \left(\frac{\log N\log \log N}{X^{1/2}}+\frac{\log N}{X^{1/2}V}  +  \frac{\log N}{X^{3/2}V^{3}}\right) \right)
\end{align*}
In the case $\sigma<0$, we choose $X$ as a sufficiently large absolute constant and $V$ as a sufficiently small absolute constant. Then we have 
\[
 X^{1/2} \log \log N +\frac{1}{X^{1/2}V}  +  \frac{1}{X^{3/2}V^{3}}  \ll \log \log N.
\]
In the case $\sigma=0$, we now choose 
\[
\alpha=1, \quad X=(\log N)^{\beta_1}(\log\log N)^{\beta_2},\quad V=(\log N)^{-\gamma_1}(\log\log N)^{\gamma_2}
\]
with $\beta_1, \gamma_1>0$ and $\beta_2, \gamma_2\in\mathbb{R}$. Then the errors are 
\begin{align*}
    &\ll X^{1/2}(\log X)^{2/3}\log\log N +V\log N +\frac{\log N}{X^{1/2}V} + \frac{\log N}{X^{3/2}V^3}.
\end{align*}
Letting $L=\log N$, the above is
\begin{align}
    &=L^{\beta_1/2} (\log L)^{5/3+\beta_2/2} +L^{1-\gamma_1} (\log L)^{\gamma_2} \notag\\
    &+L^{1-\beta_1/2+\gamma_1} (\log L)^{-\beta_2/2-\gamma_2}+L^{1-3\beta_1/2+3\gamma_1} (\log L)^{-3\beta_2/2-3\gamma_2}. \label{equation:the_error_terms-1}
\end{align}
We first determine $\beta_1$ and $\gamma_1$ such that
\[
f(\beta_1, \gamma_1) := \max( \beta_1/2, 1-\gamma_1, 1-\beta_1/2+\gamma_1, 1-3\beta_1/2+3\gamma_1)
\]
is minimized. By solving $1-\gamma_1=1-\beta_1/2+\gamma_1$ and $\beta_1/2=1-\gamma_1$, we get $(\beta_1, \gamma_1)=(4/3, 1/3)$, and hence 
\[
\min_{0<\beta_1, \gamma_1} f(\beta_1, \gamma_1) \leq f(4/3, 1/3) =2/3.
\]
Assume that $\min_{0<\beta_1, \gamma_1} f(\beta_1, \gamma_1)<2/3$. Since $\beta_1<4/3$ and $\gamma_1>1/3$, the value $1-\beta_1/2+\gamma_1$ exceeds $2/3$. However, this is a contradiction. Thus, we have $\min_{0<\beta_1, \gamma_1} f(\beta_1, \gamma_1) = 2/3$.
Applying $\beta_1=4/3$ and $\gamma_1=1/3$, the terms of \eqref{equation:the_error_terms-1} satisfy
\begin{align}
    &= L^{2/3} (\log L)^{5/3+\beta_2/2} +L^{2/3} (\log L)^{\gamma_2} +L^{2/3} (\log L)^{-\beta_2/2-\gamma_2}  +(\log L)^{-3\beta_2/2-3\gamma_2}  \notag \\
    &\ll L^{2/3} ((\log L)^{5/3+\beta_2/2} +(\log L)^{\gamma_2} +(\log L)^{-\beta_2/2-\gamma_2} ). \label{equation:the_error_terms-2}
\end{align}
We next determine $\beta_2$ and $\gamma_2$ such that
\[
g(\beta_2, \gamma_2) := \max( 5/3+\beta_2/2, \gamma_2, -\beta_2/2-\gamma_2)
\]
is minimized. By solving $5/3+\beta_2/2=\gamma_2=-\beta_2/2-\gamma_2$, we have $(\beta_2, \gamma_2)=(-20/9, 5/9)$. Similarly, $g$ is minimized as $5/9$ at $(-20/9, 5/9)$. Then the terms of \eqref{equation:the_error_terms-2} are
\[
\ll L^{2/3} (\log L)^{5/9}=(\log N)^{2/3}(\log\log N)^{5/9}.
\]
Therefore, letting $X=(\log N)^{4/3}(\log\log N)^{-20/9}$ and $V=(\log N)^{-1/3}(\log\log N)^{5/9}$, we have
\begin{align*}
    \sum_{1\leq n\leq N} \zeta(2\pi idn) \frac{e(\theta n)}{n^{1-\sigma}}
    &= d^{-\sigma}\sum_{m=1}^M \frac{1}{m^{1-\sigma}}   \sum_{H_1''<h<H_2-2d }  e(me^{(\theta-h)/d}) \\
    &+\begin{cases}
        O((\log N)^{2/3}(\log\log N)^{5/9})  &  \text{ if } \sigma=0;\\
        O(\log\log N)  &  \text{ if } \sigma<0.
    \end{cases}
\end{align*}

\begin{comment}
\begin{align*}
&\ll \log^{\beta/2} N (\log \log N)^{5/2}+ (\log^{1-\beta/2} N)(\log \log N)   \\
&\quad + (\log\log N)^2 +\log^{1-\gamma} N + \frac{\log^{1-\gamma }N}{N^{1/2}}\\
&\quad + \log^{1-\beta/2+\gamma} N  +  \log^{1-3\beta/2+3\gamma} N + \log^{1-\beta} N\\
&\ll  (\log^{\beta/2} N) (\log \log N)^{5/2}+\log^{1-\gamma} N +  \log^{1-\beta/2+\gamma} N + \log^{1-\beta} N.  
\end{align*}
By solving $1-\gamma = 1-\beta/2 + \gamma$, we have $\gamma=\beta/4$. Therefore, 
\begin{align*}
 &(\log^{\beta/2} N) (\log \log N)^{5/2}+ \log^{1-\gamma} N +  \log^{1-\beta/2+\gamma} N + \log^{1-\beta} N\\
&\ll (\log^{\beta /2} N + \log^{1-\beta/4}N)(\log \log N)^{5/2} .
\end{align*}
Therefore, by solving $\beta/2= 1-\beta/4$, we choose $\beta=4/3$. Thus, we have
\[
(\log^{\beta /2} N + \log^{1-\beta/4}N)(\log \log N)^{5/2} \ll (\log^{2/3} N) (\log \log N)^{5/2}. 
\]
and hence we obtain
\begin{align*}
\sum_{1\leq n\leq N} \frac{\zeta(2\pi idn) e(\theta n)}{n^{1-\sigma}} 
&= d^{-\sigma}\sum_{m=1}^M \frac{1}{m^{1-\sigma}}   \sum_{H_1''<h<H_2-2d }  e(me^{(\theta-h)/d})  \\
 &+ O(\log \log N + \ichi(\sigma=0) (\log^{2/3} N) (\log \log N)^{5/2})
\end{align*}
\end{comment}

%%%%%%%%%%%%%%%%%%%%%%%%%%%%%%%%%%%%%%%%%%%%%%

\section{Proof of Theorems~\ref{Theorem-Bernoulli0} and \ref{Theorem-Bernoulli1}}\label{section:Proof_of_Theorems_2.1_and_2.2}
Let $R_\sigma(N)=\log \log N + \ichi(\sigma=0) (\log N)^{2/3} (\log \log N)^{5/9}$. Let $k$ be a non-negative integer. We have 
\begin{align*}
&\sum_{1\leq |n|\leq N} \zeta(-k +idn) \frac{e(\theta n)}{n^{k+1}}\\
&= \sum_{1\leq n\leq N} \zeta(-k +idn) \frac{e(\theta n)}{n^{k+1}}  +  (-1)^{k+1}\overline{\sum_{1\leq n\leq N} \zeta(-k +idn) \frac{e(\theta n)}{n^{k+1}}}\\
&=d^k\sum_{1\leq m\leq M} \frac{1}{m^{k+1}} \sum_{H_1''(m)<h<H_2(m)-2d }  e(m e^{(\theta-h)/d} )\\
&\quad + d^k\sum_{1\leq m\leq M} \frac{1}{(-m)^{k+1}} \sum_{H_1''(m)<h<H_2(m)-2d }  e(-m e^{(\theta-h)/d} )  +O(R_{\sigma}(N) ) \\
&= d^k \sum_{1\leq |m|\leq M} \frac{1}{m^{k+1}}   \sum_{H_1''(|m|)<h<H_2(|m|)-2d }  e(m e^{(\theta-h)/d} )+O (R_\sigma(N)  ).
\end{align*}

We observe that 
\begin{align*}
H_1''(|m|)< h < H_2(|m|)-2d &\iff X e^{2} < \xi_{|m|,h} < N(1-V) \\
&\iff d e^{(h-\theta)/d}X e^{2} < |m| < d e^{(h-\theta)/d}N(1-V).
\end{align*}
Let $A= d e^{(h-\theta)/d}X e^{2}$ and $B=d e^{(h-\theta)/d}N(1-V)$.  By switching the double summations, we have 
\begin{align*}
&\sum_{1\leq |n|\leq N} \zeta(2\pi idn) \frac{e(\theta n)}{n^{k+1}} \\
&= d^k  \sum_{H''_1(1)<h<H_2(M)-2d} \sum_{\substack{ 1\leq |m|\leq M \\  A < |m| < B } } \frac{e(me^{(\theta-h)/d})  }{m^{k+1}}   +O(R_\sigma(N) ).
\end{align*}
Further, we observe that 
\begin{align} 
A\geq 1 &\iff e^{(h-\theta)/d}\geq \frac{1}{dXe^2} \nonumber\\ 
&\iff h \geq d \log \left(\frac{1}{dX}\right) +\theta -2d=H_2(1)-2d, \label{eq:equivalentA}\\ 
B \geq M &\iff e^{(h-\theta)/d} \geq \frac{M}{dN(1-V)} \nonumber\\ 
&\iff h\geq d\log \left(\frac{M}{dN(1-V)}\right) +\theta=H''_1(M). \label{eq:equivalentB}
\end{align}
By Lemma \ref{Lemma:saw-tooth-function}, for every $a,b>0$ and $x\in \mathbb{R}$, we have
\begin{equation}\label{eq:sin<<1}
\left|\sum_{a\leq |m|\leq b}  \frac{e(mx) }{m} \right|  \ll \left|\sum_{1\leq n< a}  \frac{\sin(2\pi n x) }{\pi x} \right| + \left|\sum_{1\leq n\leq  b}\frac{\sin(2\pi n x) }{\pi x} \right|  \ll1. 
\end{equation}
By $M \asymp N$ and $0<V\leq 1/4 $, there exists a constant $\lambda\geq 1$ such that 
\[
H''_1(M) \leq d\lambda + \theta.
\] 
Further, we see that
 \begin{equation}\label{eq:restrictH2-2d}
0\leq d\lambda +\theta  - (H_2(1)-2d) \ll \log X\ll \log \log N.  
\end{equation}
Combining \eqref{eq:equivalentA}, \eqref{eq:equivalentB}, \eqref{eq:sin<<1}, and \eqref{eq:restrictH2-2d} implies that
\begin{align*}
&\sum_{H_2(1)-2d\leq h<H_2(M)-2d} \sum_{\substack{ 1\leq |m|\leq M \\  A < |m| < B } } \frac{e(m e^{(\theta-h)/d}) }{m^{k+1}}\\
&= \sum_{d\lambda +\theta \leq h<H_2(M)-2d} \sum_{ A< |m|\leq M   } \frac{e(m e^{(\theta-h)/d}) }{m^{k+1}} + O(\log \log N).
\end{align*}
In the case $k>0$, it follows that 
\begin{align*}
&\sum_{d\lambda +\theta \leq h<H_2(M)-2d} \sum_{ A< |m|\leq M   } \frac{e(m e^{(\theta-h)/d}) }{m^{k+1}}\ll \sum_{d\lambda +\theta \leq h<H_2(M)-2d} X^{-k} e^{(\theta-h)k/d}\ll X^{-k}\ll 1.  
\end{align*}
In the case $k=0$, by Lemma \ref{Lemma:saw-tooth-function}, we observe that
\begin{align*}
&\sum_{d\lambda+ \theta \leq h<H_2(M)-2d} \sum_{ A< |m|\leq M   } \frac{e(m e^{(\theta-h)/d}) }{m} \ll  \sum_{d\lambda +\theta\leq h<H_2(M)-2d} \min\left(1, \frac{1}{(2A+1) \|e^{(\theta-h)/d}\| } \right) .
\end{align*}
By $\lambda\geq 1$, if $d\lambda+\theta \leq  h< H_2 (M)-2d$, then we see that $0<e^{(\theta-h)/d}\leq e^{-\lambda} \leq e^{-1}<1/2$. In addition, $A\geq 1$ by \eqref{eq:equivalentA}. Therefore, we have
\begin{align*}
\sum_{d\lambda +\theta\leq h<H_2(M)-2d} \min\left(1, \frac{1}{(2A+1) \|e^{(\theta-h)/d}\| } \right)&\ll \sum_{\theta+d \leq h<H_2(M)-2d}  \min(1, A^{-1} e^{(h-\theta)/d} ) \\
&\ll X^{-1}\log N \ll R_\sigma(N).  
\end{align*}

Thus, we obtain 
\begin{align*}
&\sum_{1\leq |n|\leq N} \frac{\zeta(2\pi idn)e(\theta n)}{n^{k+1}} = d^k \sum_{H''_1(1)<h<H_2(1)-2d } \sum_{\substack{ 1\leq |m|\leq M \\  A < |m| \leq B } } \frac{e(m e^{(\theta-h)/d}) }{m^{k+1}}   +O(R_\sigma(N) ).
\end{align*}
By \eqref{eq:equivalentA} and \eqref{eq:equivalentB}, if $H''_1(1)<h<H_2(1)-2d$, then we see that $A<1$ and $B\leq M$. Therefore, we have  
\begin{align*}
\sum_{1\leq |n|\leq N} \frac{\zeta(2\pi idn)e(\theta n)}{n^{k+1}} = d^k\sum_{H''_1(1)<h<H_2(1)-2d } \sum_{ 1\leq |m|\leq B   } \frac{e(m e^{(\theta-h)/d}) }{m^{k+1}}   +O(R_\sigma(N) ).
\end{align*}
In the case $k>0$, the Fourier expansion of $\psi_{k+1}(x)$ (recall \eqref{eq:FourierExp}) leads to 
\[
 \sum_{ 1\leq |m|\leq B   } \frac{e(m e^{(\theta-h)/d}) }{m^{k+1}}  = -\frac{(2\pi i)^{k+1}}{(k+1)!}\psi(e^{(\theta-h)/d} )  + O( B^{-k} ). 
\]
By $B=de^{(h-\theta)/d}N(1-V)$, we have 
\begin{align*}
&\sum_{1\leq |n|\leq N} \frac{\zeta(2\pi idn)e(\theta n)}{n^{k+1}} =  -\frac{(2\pi i)^{k+1} d^k}{(k+1)!} \sum_{H''_1(1)<h<H_2(1)-2d } \psi_{k+1} (e^{(\theta-h)/d} )  \\
&+ O \left(\sum_{H''_1(1)<h<H_2(1)-2d } N^{-k} e^{(\theta-h)k/d} \right) +O(R_\sigma(N) ),
\end{align*}
and the errors are upper bounded by $\ll 1 + \log\log N \ll \log\log N$. Moreover, since $H_1''(1)<h<H_2(1)-2d$ if and only if $d\log (dX)-\theta \leq -h \leq d \log (d(1-V)N) -\theta$, by replacing $h$ with $-h$, we have
\[
\sum_{H''_1(1)<h<H_2(1)-2d } \psi_{k+1} (e^{(\theta-h)/d} ) = \sum_{0<h <d\log N }  \psi_{k+1} (e^{(\theta+h)/d} ) + O(\log \log N).
\]
This completes the proof of Theorem~\ref{Theorem-Bernoulli1}. 

In the case $k=0$, it follows that
\begin{align*}
\sum_{1\leq |n|\leq N} \zeta(2\pi idn) \frac{e(\theta n)}{n} = \sum_{H''_1(1)<h<H_2(1)-2d } \sum_{ 1\leq |m|\leq B   } \frac{e(m e^{(\theta-h)/d}) }{m}   +O(R_\sigma(N) ).
\end{align*}

By Lemma \ref{Lemma:saw-tooth-function}, $\sum_{ 1\leq |m|\leq B   } e(m e^{(\theta-h)/d}) /m\ll 1$, and hence 
\[
\sum_{H''_1(1)<h<H_2(1)-2d } \sum_{ 1\leq |m|\leq B   } \frac{e(m e^{(\theta-h)/d}) }{m}  = \sum_{0<h < d\log N}  \sum_{ 1\leq |m|\leq B'   } \frac{e(m e^{(\theta+h)/d}) }{m} +O(\log\log N),
\]
where $B'=de^{-(\theta+h)/d}N(1-V)$. Furthermore, if $de^{-(\theta+h)/d}N-B'\geq1$, then $h$ does not exceed $d\log(dVN)-\theta$. Thus, we have
\begin{align*}
    &\sum_{0<h\leq d\log(dVN)-\theta } \sum_{B' \leq |m|\leq de^{-(\theta+h)/d}N  } \frac{e(m e^{(\theta+h)/d}) }{m} \\
    &\ll \sum_{0<h\leq d\log(dVN)-\theta} B'^{-1}(de^{-(\theta+h)/d}N -B') \ll V\log N\ll R_\sigma(N).
\end{align*}
We evaluate the remainder term as
\begin{align*}
    &\sum_{d\log(dVN)-\theta < h < d\log N} \sum_{B' \leq |m|\leq de^{-(\theta+h)/d}N  } \frac{e(m e^{(\theta+h)/d}) }{m}\\
    &\qquad\ll \sum_{d\log(dVN)-\theta < h < d\log N} 1 \ll \log\log N.
\end{align*}

Therefore, we conclude that
\begin{align*}
\sum_{1\leq |n|\leq N} \zeta(2\pi idn) \frac{e(\theta n)}{n} &= \sum_{0<h<d\log N } \sum_{ 1\leq |m|\leq de^{-(\theta+h)/d} N } \frac{e(m e^{(\theta+h)/d}) }{m}   \\
&\quad +O((\log N)^{2/3} (\log \log N)^{5/9} ).
\end{align*}
This completes the proof of Theorem~\ref{Theorem-Bernoulli0}, where \eqref{eq:thm-Bernoulli0-2} follows from Lemma \ref{Lemma:saw-tooth-function} and $|\sin \pi x| \geq 2 \|x\|$. \\

\section*{Acknowledgment}
The first author was financially supported by JST SPRING, Grant Number JPMJSP2125. The second author was financially supported by JSPS KAKENHI Grant Numbers JP22J00025 and JP22KJ0375. We would like to thank Professor Kohji Matsumoto for useful advice.

\end{document}